\documentclass[reqno]{amsart}

\usepackage{amsmath,amssymb,amsthm}

\usepackage{tikz}

\usepackage{caption}
\usepackage{graphicx}
\usepackage{graphics}
\usepackage{algorithm}
\usepackage{algorithmicx}
\usepackage{algpseudocode}

\newtheorem{theorem}{Theorem}

\newtheorem{lemma}{Lemma}

\newtheorem{proposition}{Proposition}

\newcommand{\cdf}{\operatorname{CDF}}
\newcommand{\sign}{\operatorname{sign}}

\begin{document}

\title[]{Robust Online Sampling from Possibly\\ Moving Target Distributions}

\author[]{Fran\c{c}ois Cl\'ement \and Stefan Steinerberger}

\address{Department of Mathematics, University of Washington, Seattle}
 \email{fclement@uw.edu }
 \email{steinerb@uw.edu}

\begin{abstract} 
We suppose we are given a list of points $x_1, \dots, x_n \in \mathbb{R}$, a target probability measure $\mu$ and are asked to add
additional points $x_{n+1}, \dots, x_{n+m}$ so that $x_1, \dots, x_{n+m}$ is as close as possible to the distribution of $\mu$; additionally, we want this to be true uniformly for all $m$. We propose a simple method that achieves this goal. It selects new points in regions where the existing set is lacking points and avoids regions that are already overly crowded. If we replace $\mu$ by another measure $\mu_2$ in the middle of the computation, the method dynamically adjusts and allows us to keep the original sampling points. $x_{n+1}$ can be computed in $\mathcal{O}(n)$ steps and we obtain state-of-the-art results. It appears to be an interesting dynamical system in its own right; we analyze a continuous mean-field version that reflects much of the same behavior. 
\end{abstract}

\subjclass[2020]{}

\maketitle

\vspace{-0pt}

\section{Introduction and Results}

\subsection{The problem}
We consider the setting where we have a (known) probability measure $\mu$ supported on $\mathbb{R}$ as well as a finite number of points $x_1, \dots, x_n \in \mathbb{R}$. We want to add additional points $x_{n+1}, x_{n+2}, \dots, x_{n+m}$, the goal is to end up with a finite set of points $x_1, \dots, x_{n+m}$ that approximates the measure $\mu$. Since the problem setup is one-dimensional, this problem allows for a relatively easy solution when $m$, the number of points to be added, is fixed: one can place them easily in a way to minimize virtually any metric. We are interested in the setting when $m$ itself is \textit{not known}. The assumption of $m$ being unknown means we want to have \textit{uniformly} good approximation of the prescribed measure $\mu$ for all $m$.

\vspace{-10pt}

\begin{center}
    \begin{figure}[h!]
    \begin{tikzpicture}
        \node at (0,0) {\includegraphics[width=0.65\textwidth]{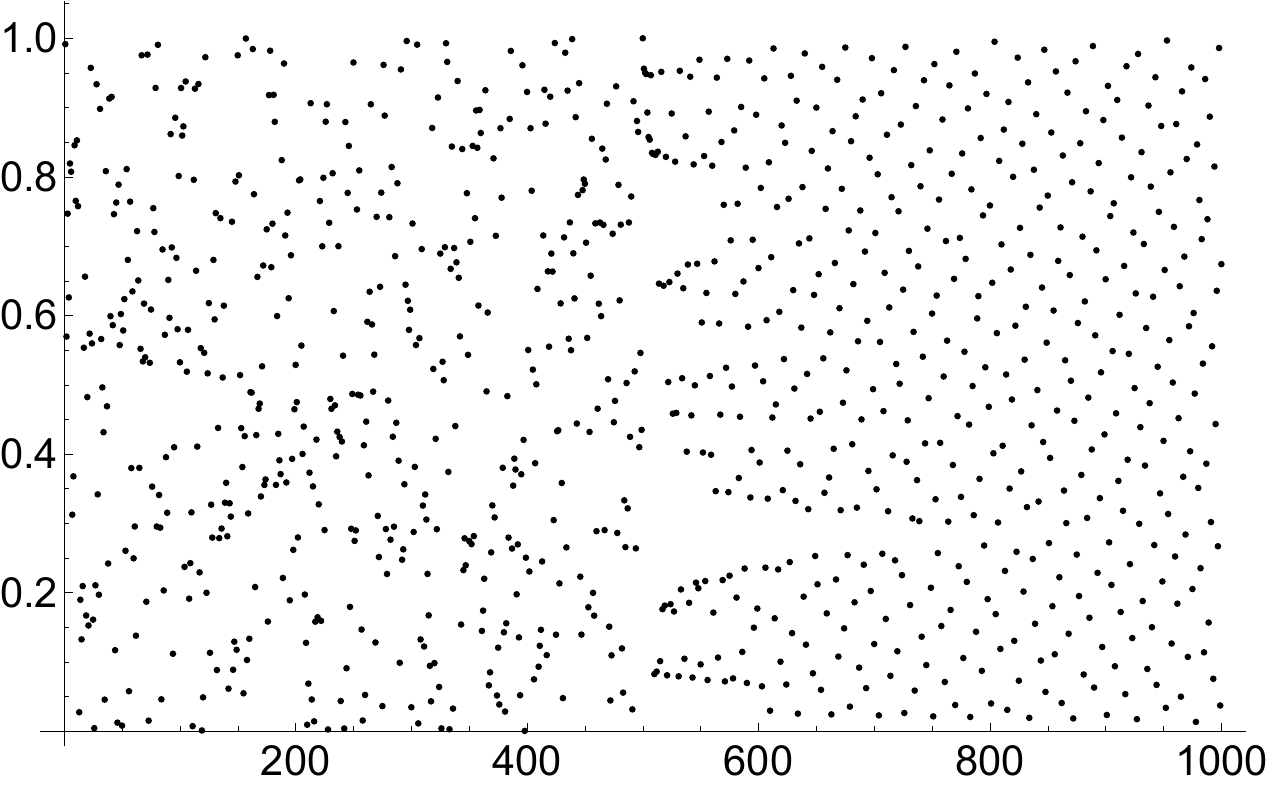}};
    \end{tikzpicture}
    \caption{500 uniformly distributed random variables followed by 500 new points: they initially fix gaps before, around point $\sim 700$, reaching an equilibrium and refining at smaller scales. }
    \end{figure}
\end{center}

Besides being of intrinsic interest, there are several applications.
\begin{enumerate}
    \item We had an initial sampling budget but based on first results would like to sample more points. This could be the Bayes setting, where we have been sampling with respect to a measure $\mu_1$ and found out that we really should sample with respect to a nearby measure $\mu_2$.
    \item We are uncertain about the total sampling budget, or may wish to stop the experiment early based on partial results. Initial samples must still be as uniformly spread out as possible. This is in particular the case in hyperparameter tuning~\cite{Teyt,Hyperband} and Bayesian optimization~\cite{Garnett}. 
    \item We have a set of sampling procedures which individually require samples to be uniformly distributed, and that keep diversity relative to each other. This is for example the case in sequential iterations of CMA-ES~\cite{DiedOpti} or in batch sampling~\cite{Gonza}.
\end{enumerate}

We note that the problem of \textit{uniformly} distributing points is already difficult if we start with no points at all: finding a sequence of points $(x_k)_{k=1}^{\infty}$ with the property that $\left\{x_1, \dots, x_n \right\}$ is very evenly distributed in the unit interval $[0,1]$ \textit{uniformly} in $n$ is a nontrivial problem (for which we obtain state-of-the-art results, see \S 5.2).

\begin{center}
    \begin{figure}[h!]
    \begin{tikzpicture}
        \node at (0,0) {\includegraphics[width=0.6\textwidth]{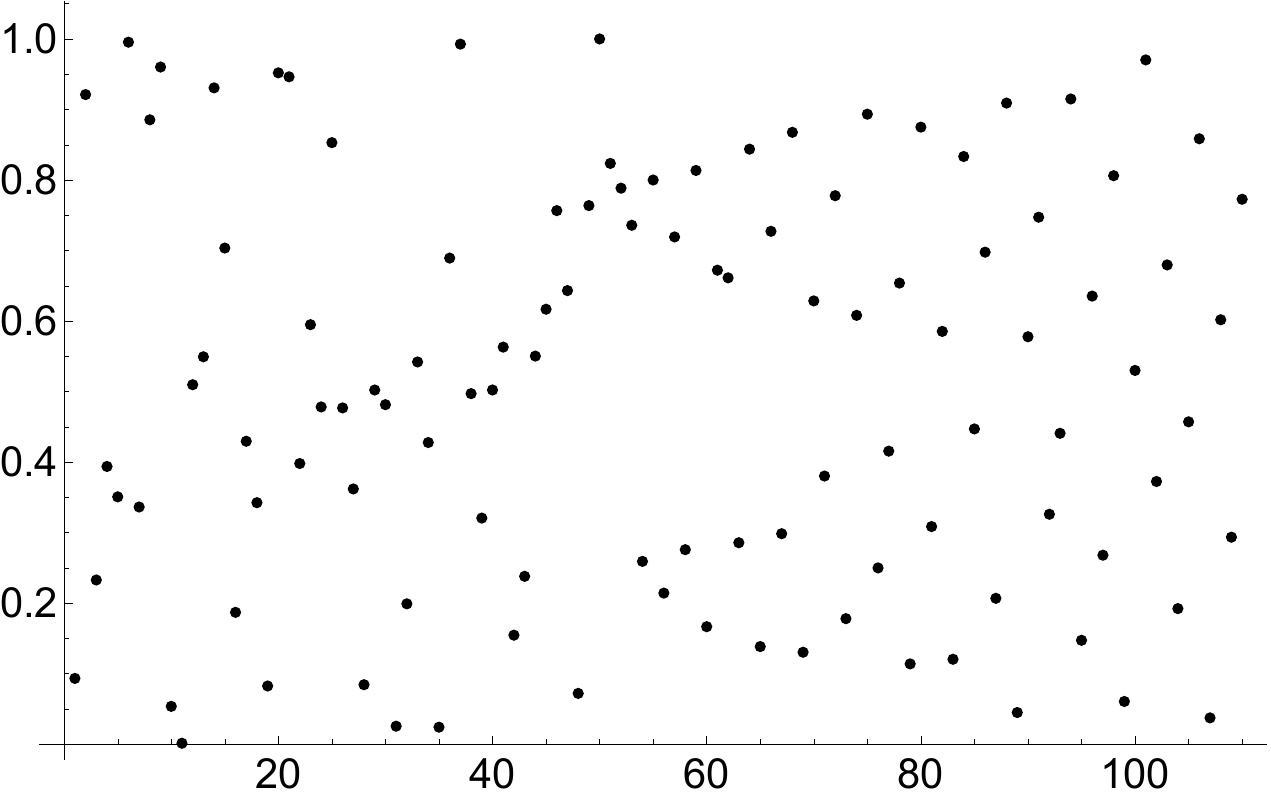}};
        \node at (-4, -4) {\includegraphics[width=0.2\textwidth]{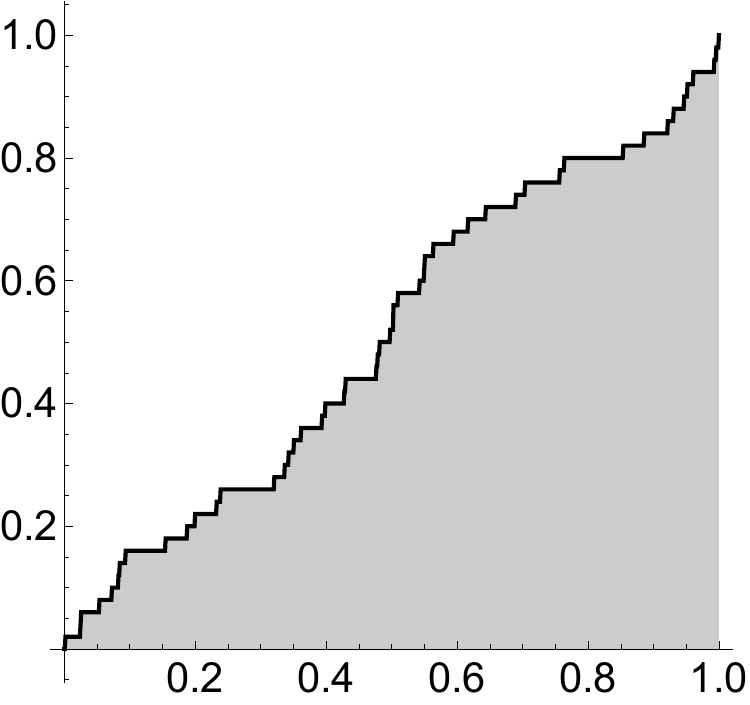}};
        \node at (-4, -5.5) {after 50};
        \node at (-1, -4) {\includegraphics[width=0.2\textwidth]{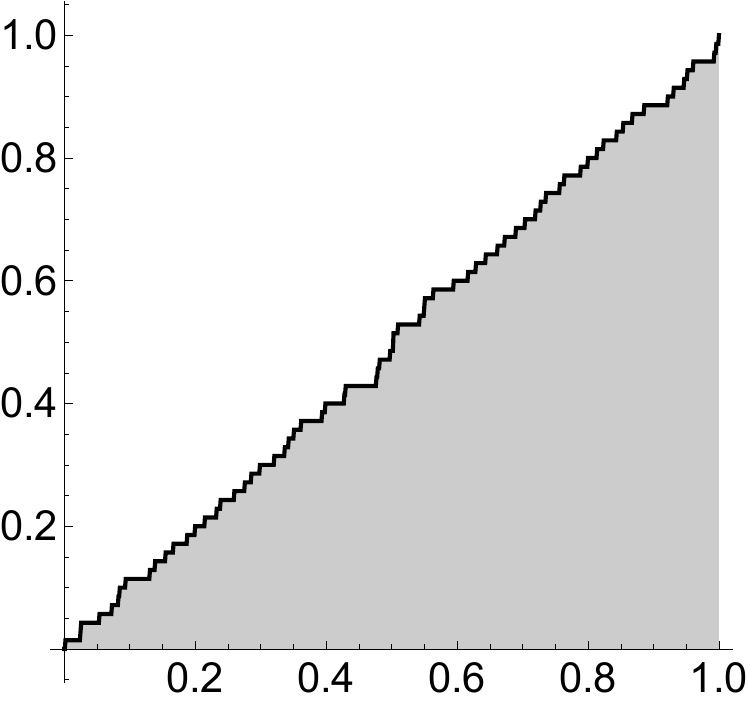}};
            \node at (-1, -5.5) {after 70};
        \node at (2, -4) {\includegraphics[width=0.2\textwidth]{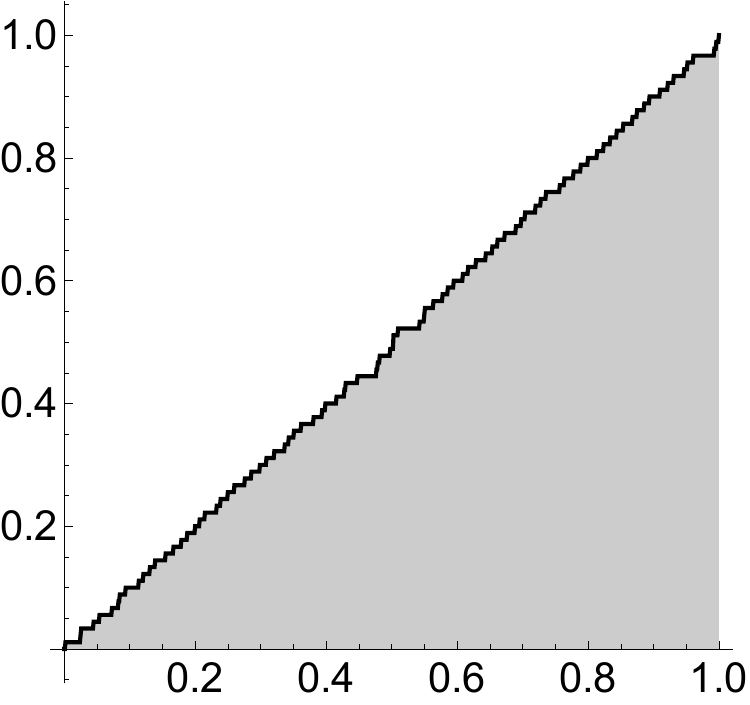}};
           \node at (2, -5.5) {after 90};
        \node at (5, -4) {\includegraphics[width=0.2\textwidth]{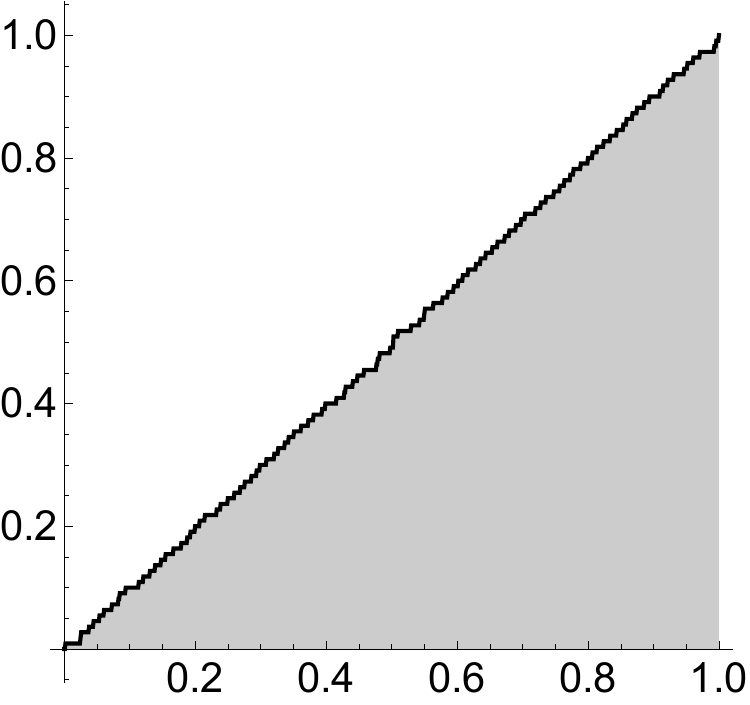}};
           \node at (5, -5.5) {after 110};
    \end{tikzpicture}
    \caption{50 iid random points in $[0,1]$ followed by 60 new points. The new points fill the gaps, the CDF converges quickly to $x$.}
    \end{figure}
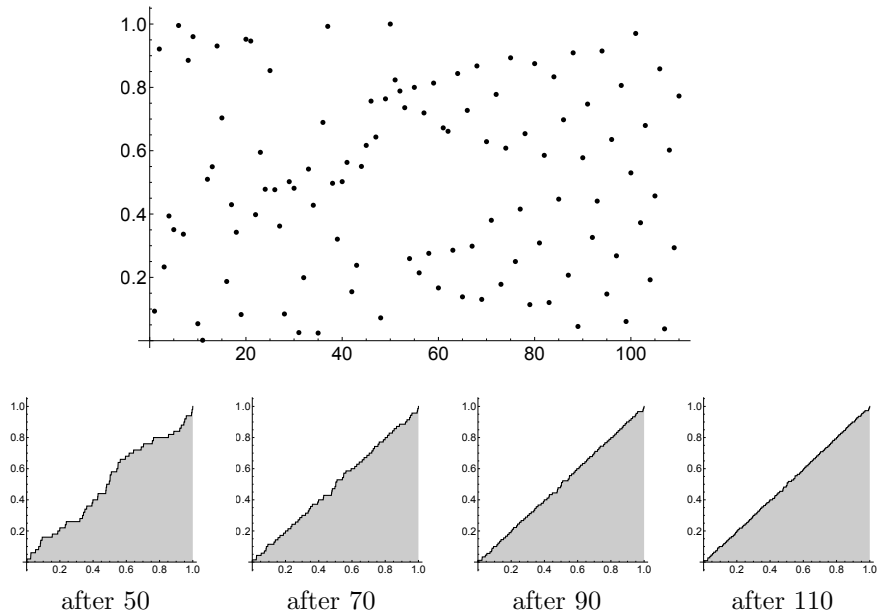
\end{center}

\vspace{-20pt}

\subsection{The method.}
The main purpose of this paper is to present a simple algorithm that performs exceedingly well (see \S 5.2). Using the cumulative distribution function of the target measure, it is enough to solve the problem for the uniform distribution on $[0,1]$. The case of general measures is discussed in Fig. \ref{fig:extend} and \S 1.3.
To solve the problem for the uniform measure on $[0,1]$, we define a notion of `energy' for any set of $n$ points in the unit interval via
$$ E(\mathbf{x}) = \min_{\pi \in S_n} \sum_{i=1}^{n}~ \left| x_{\pi(i)} - \frac{i}{n} \right|,$$
where $\pi:\left\{1,2,\dots, n\right\} \rightarrow \left\{1,2,\dots, n\right\}$ ranges over all permutations.
The role of the permutation is to sort the points; if the points are already ordered, $x_1 \leq x_2 \leq \dots \leq x_n$, then the energy simplifies to
$$ E(\mathbf{x}) =\sum_{i=1}^{n}~ \left| x_{i} - \frac{i}{n} \right|.$$
We propose adding $x_{n+1}$ so that it minimizes the energy $E(x_1, \dots, x_n, x_{n+1})$, i.e.
$$ x_{n+1} = \arg \min_{0 \leq x \leq 1} E(x_1, \dots, x_n, x).$$
Greedy methods of this type are easy to implement. Greedy methods tend to work for a while but usually do not perform very well in the long run: errors and mistakes accumulate over time, they have a tendency to get stuck in strange places in the configuration space. Somewhat to our surprise, this does not appear to be the case here; the arising method is simple, fast and completely robust. The robustness stems from an underlying variational structure that we analyze for a continuous limit (see \S 1.6). Moreover, the minimum is easy to find and easy to store because it is a rational number with a priori known denominator.

\begin{theorem} Given $0 \leq x_1 \leq x_2 \leq \dots \leq x_n \leq 1$, then
$$ x_{n+1} = \arg \min_{0 \leq x \leq 1} E(x_1, \dots, x_n, x)$$
can be computed using $\mathcal{O}(n)$ operations. Moreover, $x_{n+1} = k/(n+1)$ is a rational number with $1 \leq k \leq n+1$. If $1 \leq k \leq n-1$, then
$x_k < (k+1)/(n+1) < x_{k+1}$.
\end{theorem}

The result tells us that the optimization problem can be solved fairly quickly without any sophisticated numerical methods, the result is easy to store and reuse. Finally, the new point is either going to be at the very beginning, $1/(n+1)$, or at the very end, $1$, or it is going to end up in a location $x \sim k/(n+1)$ so that $[0,x]$ contains $k$ out of $n+1$ points.

 \begin{center}
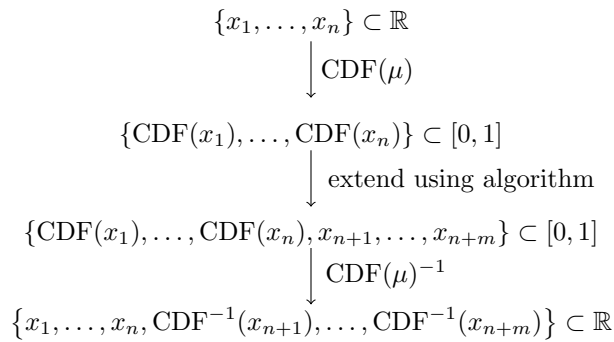
\begin{figure}[h!]
\begin{tikzpicture}
    \node at (0,0) {$ \left\{x_1, \dots, x_n \right\} \subset \mathbb{R}$};
    \draw [->] (0, -0.25) -- (0, -1); 
    \node at (0.75, -0.65) {CDF($\mu$)};
      \node at (0,-1.5) {$ \left\{\cdf(x_1), \dots, \cdf(x_n) \right\} \subset [0,1]$};
    \draw [->] (0, -1.7) -- (0, -2.45); 
    \node at (2, -2.1) {extend using algorithm};
       \node at (0,-2.8) {$ \left\{\cdf(x_1), \dots, \cdf(x_n), x_{n+1}, \dots, x_{n+m} \right\} \subset [0,1]$};   
       \draw [->] (0, -3) -- (0, -3.75); 
           \node at (1, -3.35) {CDF$(\mu)^{-1}$};
       \node at (0,-4) {$ \left\{x_1, \dots, x_n, \cdf^{-1}(x_{n+1}), \dots, \cdf^{-1}(x_{n+m}) \right\} \subset \mathbb{R}$};  
\end{tikzpicture}
\caption{Using the method to extend a set of points on $\mathbb{R}$ with respect to an arbitrary probability measure $\mu$.}
\label{fig:extend}
\end{figure}
\end{center}

\subsection{General distributions} The solution for the uniform distribution in $[0,1]$ immediately generalizes to all other probability measures $\mu$.
 If one is given $x_1, \dots, x_n \in \mathbb{R}$ and if the points are supposed to   approximate the probability measure $\mu$, then we may use the cumulative distribution function $\cdf(\mu)$, apply it to $x_1, \dots, x_n$ and then work in the unit interval.
 The algorithm is outlined in Figure \ref{fig:extend}, an explicit example is given in Fig. \ref{fig:gauss}.

\begin{center}
    \begin{figure}[h!]
    \begin{tikzpicture}
        \node at (0,0) {\includegraphics[width=0.6\textwidth]{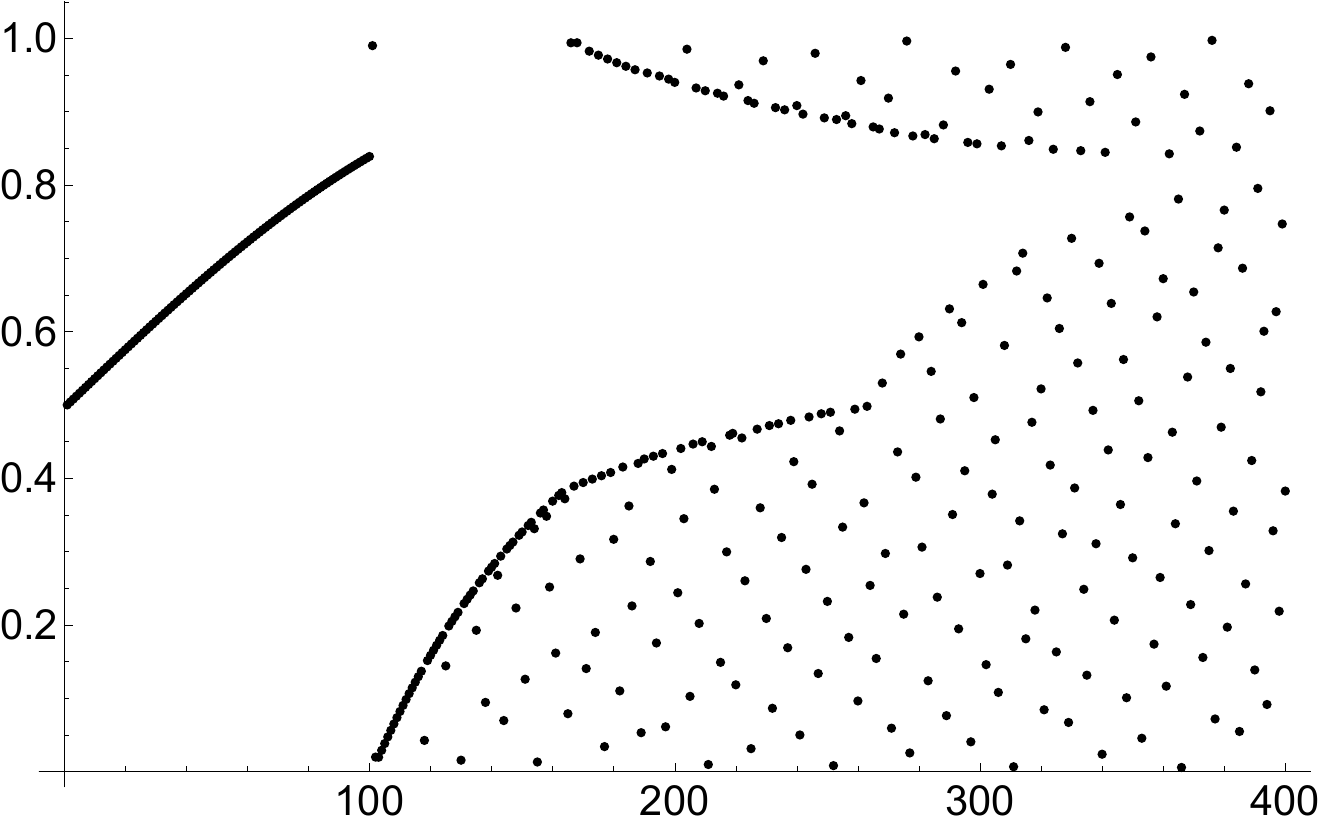}};
          \node at (6,1) {\includegraphics[width=0.38\textwidth]{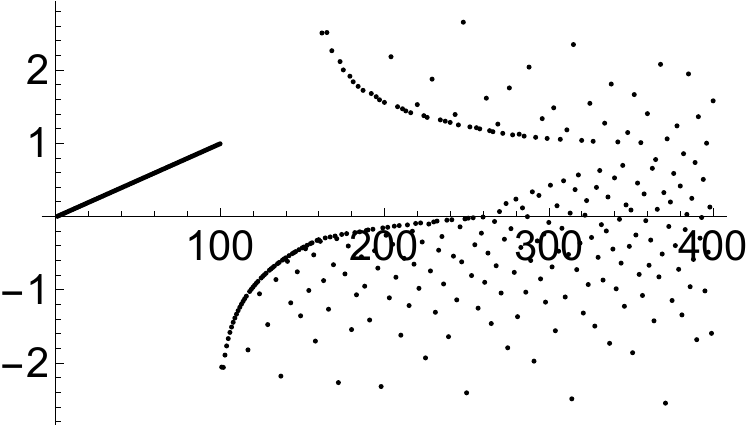}};
      \node at (6,-1.5) {\includegraphics[width=0.2\textwidth]{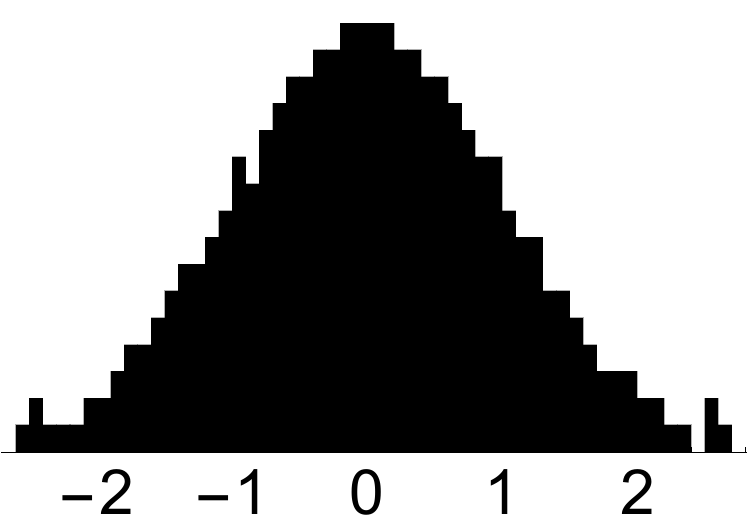}};

    \end{tikzpicture}
    \caption{(Left): the algorithm working to fix things in $\cdf-$space. (Right, top): the new points being added in real space. (Right, bottom): the final set of 400 points.}
      \label{fig:gauss}
    \end{figure}
\end{center}
 
 Suppose we try to approximate the uniform distribution in $[0,1]$ with 100 points: this is easy, we use $0/100, 1/100, 2/100, \dots, 99/100$. Imagine we sample in these points and then realize that the goal is to approximate the standard Gaussian $\mathcal{N}(0,1)$ instead. Is there a way to add new points so that we end up with a Gaussian distribution without throwing away any of the points we already have?  Note that there is a severe imbalance: the Gaussian distribution has only 
 $$ \mbox{a fraction of} \quad \int_0^1 \frac{1}{\sqrt{2\pi}} e^{-x^2/2} ~dx \sim 0.341\dots \quad \mbox{of all its points in}~[0,1].$$
This means that unless we have at least $100/0.341 \sim 293$ points, i.e. we add at least 193 points, this severe imbalance cannot be corrected. The situation is a bit more complicated: the Gaussian restricted to $[0,1]$ is far from uniform, so even more points will be required.
The behavior of the algorithm is shown in Figure \ref{fig:gauss}.
 We use the cumulative distribution function to map the points to a set of points in $[\Phi(0), \Phi(1)] = [0.5, 0.84] \subset [0,1]$ and run the method in the unit interval.  
We observe that in `$\cdf-$space', we reach a nice uniform distribution when adding another $\sim 250$ points. After that, the uniform distribution is then maintained and even further refined. The uniform distribution in $[0,1]$ in $\cdf-$space is mirrored by a high quality approximation of the Gaussian in real space.

\subsection{Where is $x_{n+1}$ going to end up?}
A single step of the algorithm is elementary but $x_{n+1}$ is not easy to predict. The purpose of this section is to explain, roughly, how $x_{n+1}$ behaves and where we can expect to find it. We introduce the discrepancy function
$$ \Delta(x) = \frac{\# \left\{1 \leq i \leq n: x_i \leq x \right\}}{n} - x.$$
It measures, approximately, whether there are `too many' or `not enough' points in $[0,x]$ compared to what one would expect from the uniform distribution. The second is a one-parameter family of functions $h_x:[0,1] \rightarrow [0,1]$ defined by
$$ h_x(y) = \begin{cases} y \qquad &\mbox{if}~y \leq x \\ 
y-1 \qquad &\mbox{if}~y > x. \end{cases}$$
\vspace{-15pt}
\begin{center}
    \begin{figure}[h!]
    \begin{tikzpicture}[scale=0.65]
       \draw [thick] (0,0) -- (3,0); 
       \draw [very thick] (0, -0.1) -- (0, 0.1);
        \draw [very thick] (3, -0.1) -- (3, 0.1);    
      \node at (0, -0.4) {0};  
     \node at (3, -0.4) {1};  
     \draw [very thick] (0,0) -- (1,1);
      \draw [very thick]   (1,-2) -- (3,0);
          \draw [thick] (5+0,0) -- (5+3,0); 
       \draw [very thick] (5+0, -0.1) -- (5+0, 0.1);
        \draw [very thick] (5+3, -0.1) -- (5+3, 0.1);    
      \node at (5+0, -0.4) {0};  
     \node at (5+3, -0.4) {1};  
     \draw [very thick] (5+0,0) -- (5+2,2);
      \draw [very thick]   (5+2,-1) -- (5+3,0);   
     \node at (1.5, -0.5) {$h_{1/3}(y)$};
     \node at (6.5, -0.5) {$h_{2/3}(y)$};
    \end{tikzpicture}
    \end{figure}
\end{center}
\vspace{-20pt}
With these two definitions in place, we can `to first order' (see \S 2.3) predict that the next point will be added so as to maximize a certain integral 
 $$ x_{n+1} \qquad \sim \qquad \arg\max_{0 \leq x \leq 1}  \quad \int_0^1 \sign\left[\Delta(y)\right] \cdot h_x(y) dy.$$
We quickly illustrate this for 4 examples of 1000 random points in the unit interval.

\begin{center}
    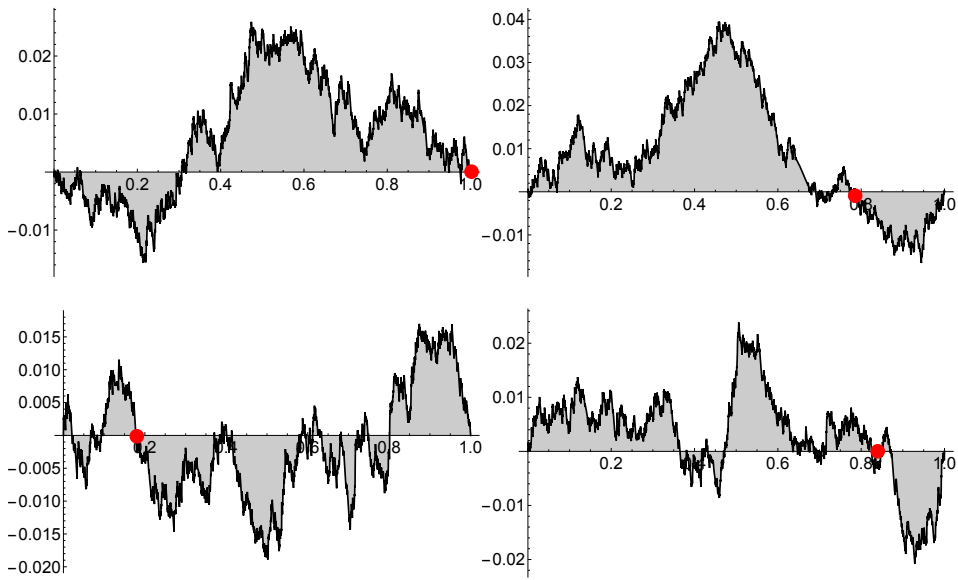
\begin{figure}[h!]
    \begin{tikzpicture}
    \node at (0,0) {\includegraphics[width=0.5\textwidth]{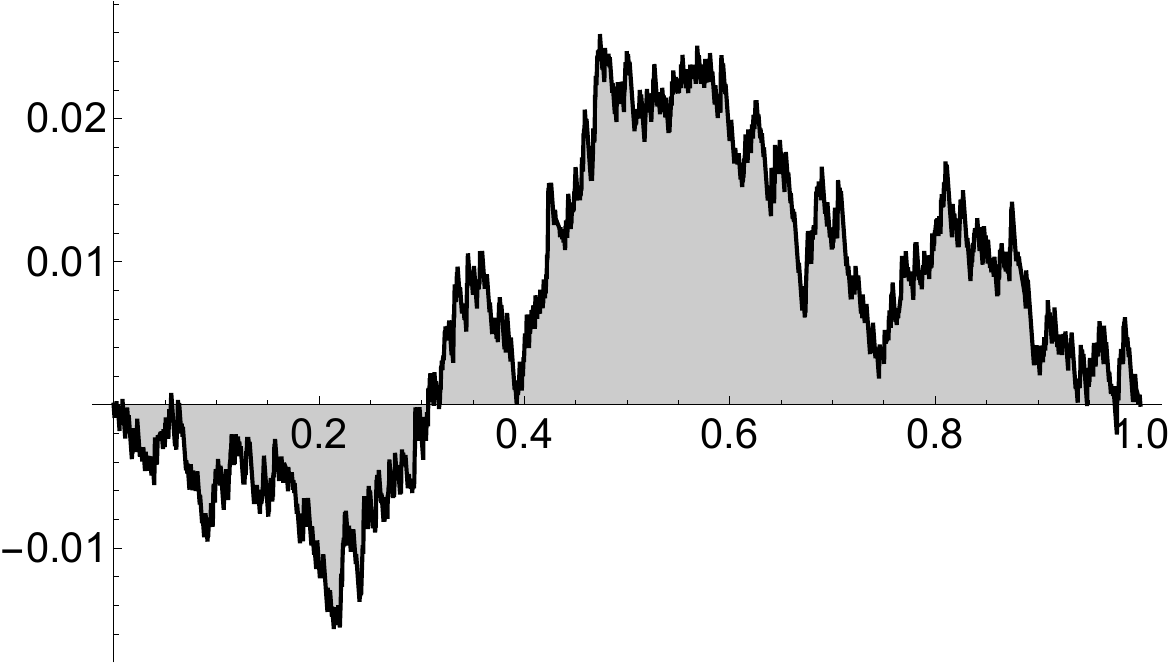}};
    \filldraw [red] (3,-0.38) circle (0.09cm);
    \node at (6.3,0) {\includegraphics[width=0.5\textwidth]{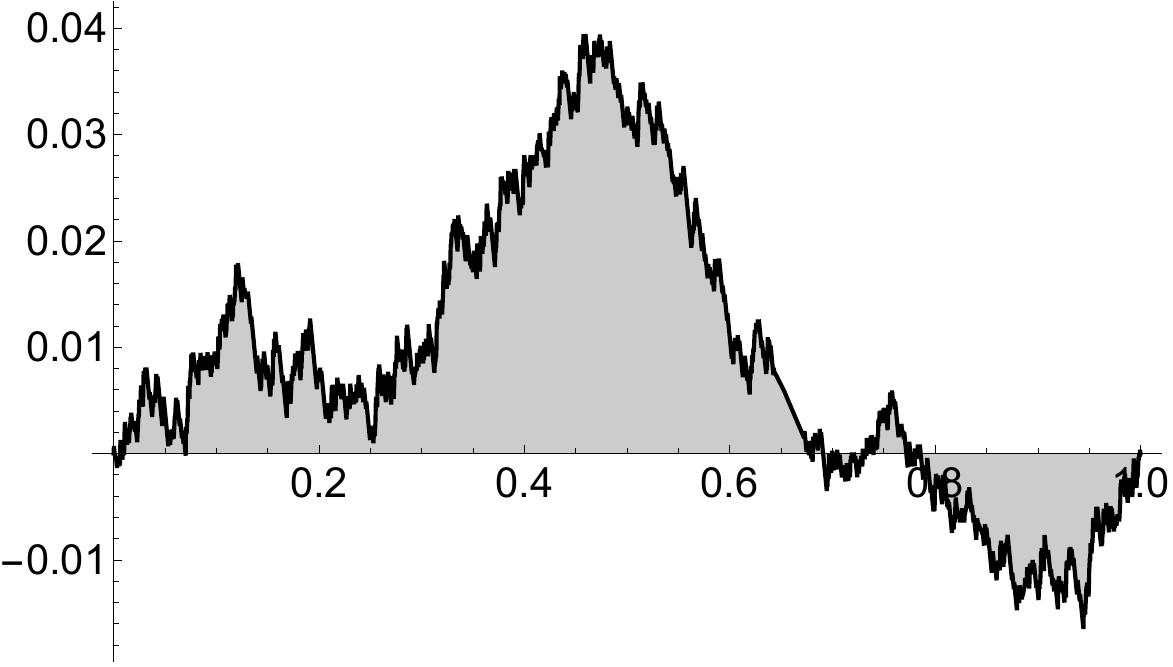}};
      \filldraw [red] (8.1,-0.7) circle (0.09cm);
    \node at (0,-4) {\includegraphics[width=0.5\textwidth]{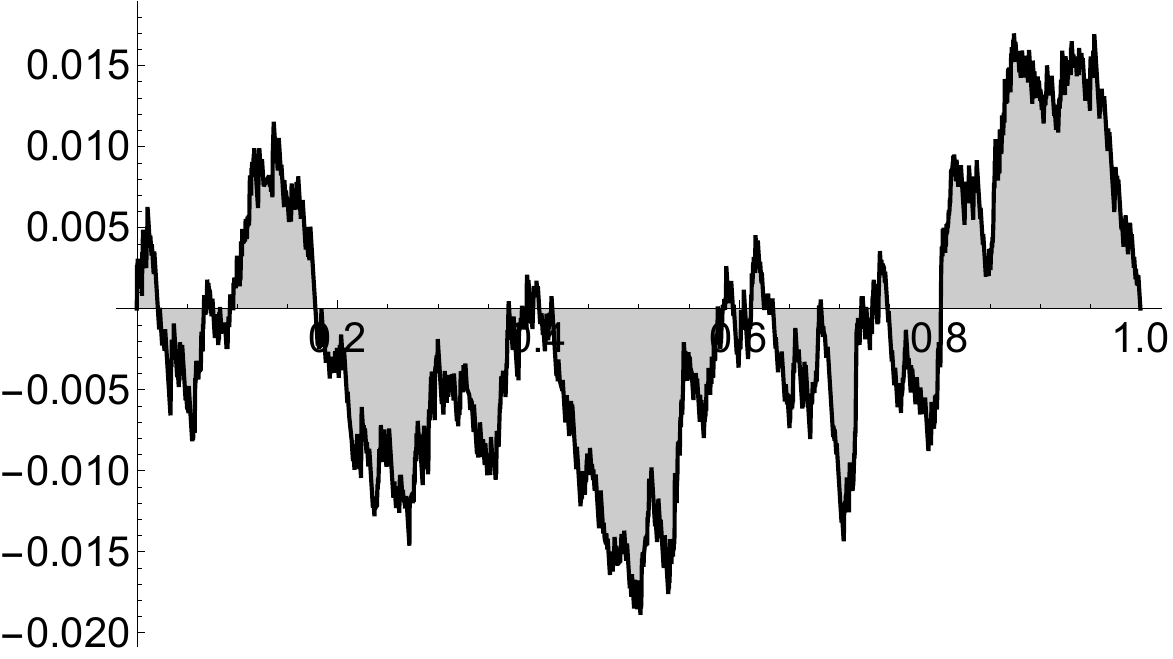}};
        \filldraw [red] (-1.45,-3.9) circle (0.09cm);
    \node at (6.3,-4) {\includegraphics[width=0.5\textwidth]{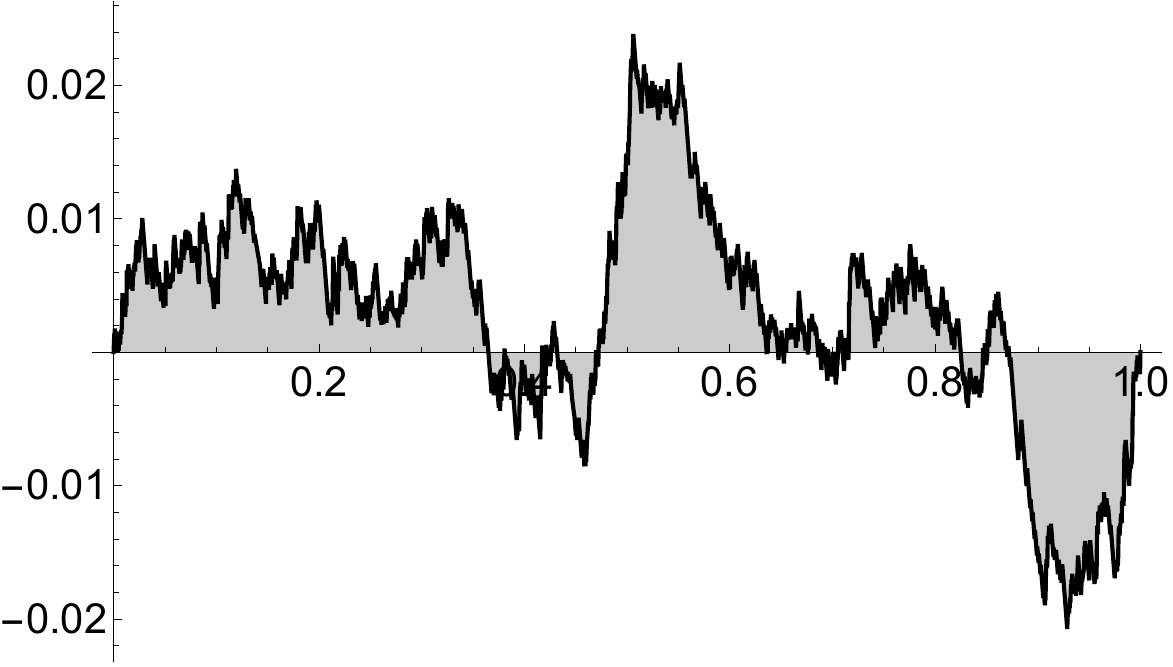}};
   \filldraw [red] (8.4,-4.1) circle (0.09cm);
    \end{tikzpicture}
    \caption{Test your intuition: four examples of 1000 random points, their discrepancy function $\Delta(x)$ (black) and $x_{1001}$ (red).}
    \end{figure}
\end{center}

\subsection{Behavior of the energy} When placing $x_{n+1}$, we do not consider the actual numerical value of the energy, we only minimize it. As it turns out, the behavior of the energy $E(x_1, \dots, x_n)$ as a sequence in $n$ is quite interesting.  Note that if $x_i = i/n$, then the energy is identically 0. However, in that case it will not be 0 in the next step. On average, the energy is at least size $\gtrsim 1$.

\begin{proposition}
    For any $x_1, \dots, x_n, x \in [0,1]$,  we have
    $$ \frac{E(x_1, \dots, x_n) + E(x_1, \dots, x_n, x)}{2} \geq \frac{1}{8}.$$
\end{proposition}
Empirically (see Fig. \ref{fig:numerics1}) this is the correct scale. To make matters more interesting, we also know that the energy level cannot be uniformly bounded. This follows from a deep result of Hal\'asz \cite{halasz} and some classic arguments (see also Graham \cite{graham}).

\begin{theorem} There exists $c>0$ such that if $(x_k)_{k=1}^{\infty}$ is \emph{any} infinite sequence in $[0,1]$, then there are infinitely many $n \in \mathbb{N}$ such that
$$\min_{\pi \in S_n} \sum_{i=1}^{n}~ \left| x_{\pi(i)} - \frac{i}{n} \right|  \geq c \sqrt{\log{n}}.$$
\end{theorem}
A nontrivial rate of this type indicates a level of unavoidable complexity and seems to rule out certain naive approaches when trying to understand the behavior of the energies of the greedy sequence.
It is known that the rate $\sqrt{\log{n}}$ is attained by some sequences \cite{steinwass}. We do not know whether the greedy procedure also attains that rate. However, is easy to do a numerical simulation.  Starting with the set $\left\{1/3, 1/2\right\}$, we computed the first $500\,000$ elements and their energies.
\begin{center}
    \begin{figure}[h!]
    \begin{tikzpicture}
\node at (0,0) {\includegraphics[width=0.4\textwidth]{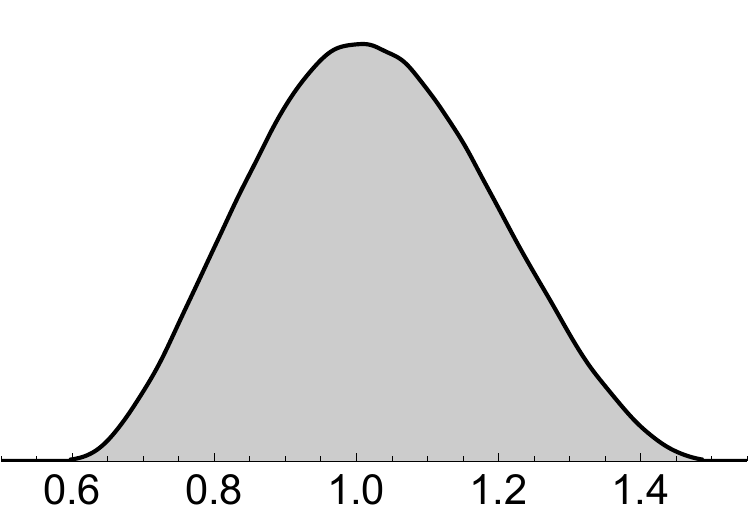}};
\node at (6,0) {\includegraphics[width=0.3\textwidth]{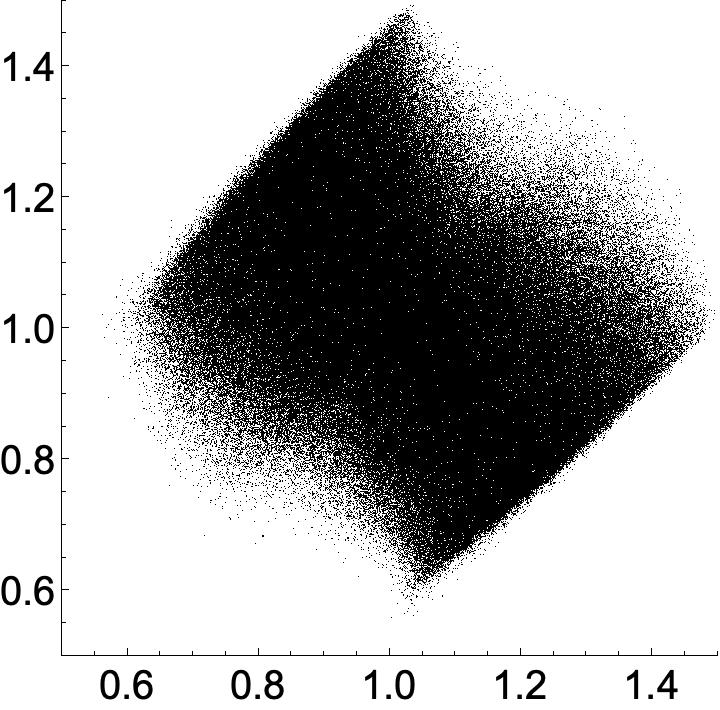}};
    \end{tikzpicture}
    \caption{Left: distribution of the first $500.000$ energies. Right: distribution of pairs of consecutive energies $(E_n, E_{n+1})_{n=1}^{5 \cdot 10^5}$.}
    \label{fig:numerics1}
    \end{figure}
\end{center}

 The energies are all in the range $[0.5, 1.5]$ and seem to be essentially stationary: $\sqrt{\log{n}}$ grows very slowly and this setting is an example of that. Plotting the distribution of two consecutive energy values $(E_{n}, E_{n+1})$ in the $xy-$plane shows an intricate structure: small values tend to be followed by fairly large values, while fairly large values tend to be followed by small values (see Fig. \ref{fig:numerics1}).
To make matters even more intriguing, if we plot three consecutive energies, that is $(E_n, E_{n+1}, E_{n+2})$ as a point in three dimensions, we observe a clear separation into two distinct clusters.  A projection onto a suitable one-dimensional axis shows an approximately bimodal behavior that seems to roughly describe the two `modes' (small, large, small) and (large, small, large). This is evidence of intricate dynamical behavior. We do not know whether this is what one might call a `small $n$ phenomenon', something that disappears asymptotically, or whether there is a more concrete Theorem lurking in the background. 

\begin{center}
    \begin{figure}[h!]
    \begin{tikzpicture}
\node at (-1,0) {\includegraphics[width=0.35\textwidth]{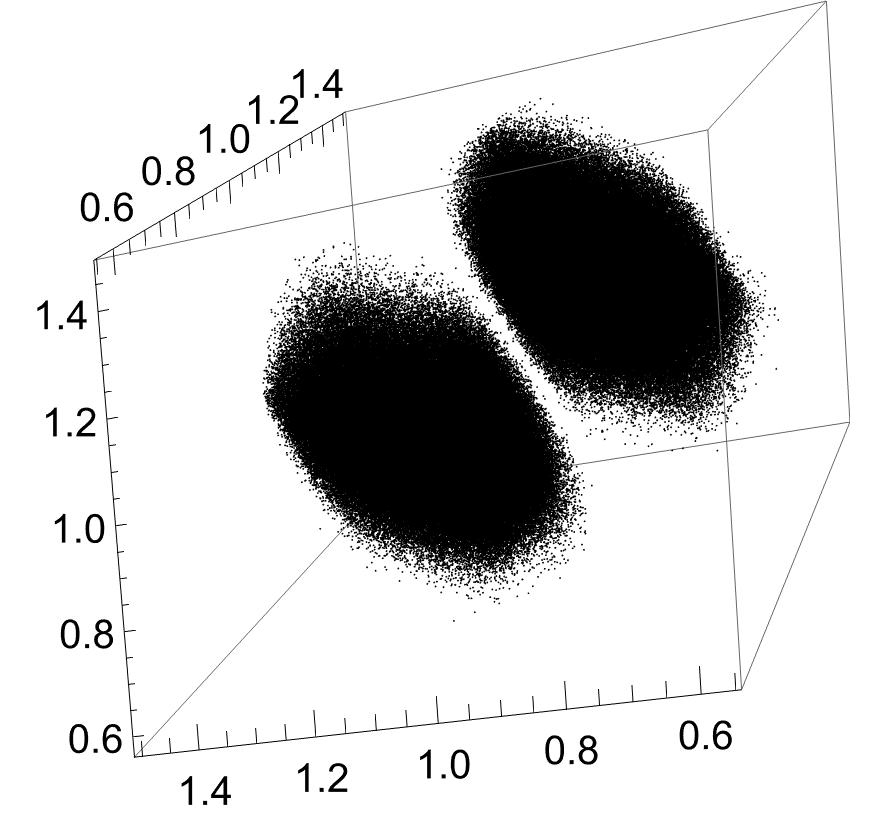}};
\draw [->] (2,0) -- (4,0);
\node at (6.8,0) {\includegraphics[width=0.35\textwidth]{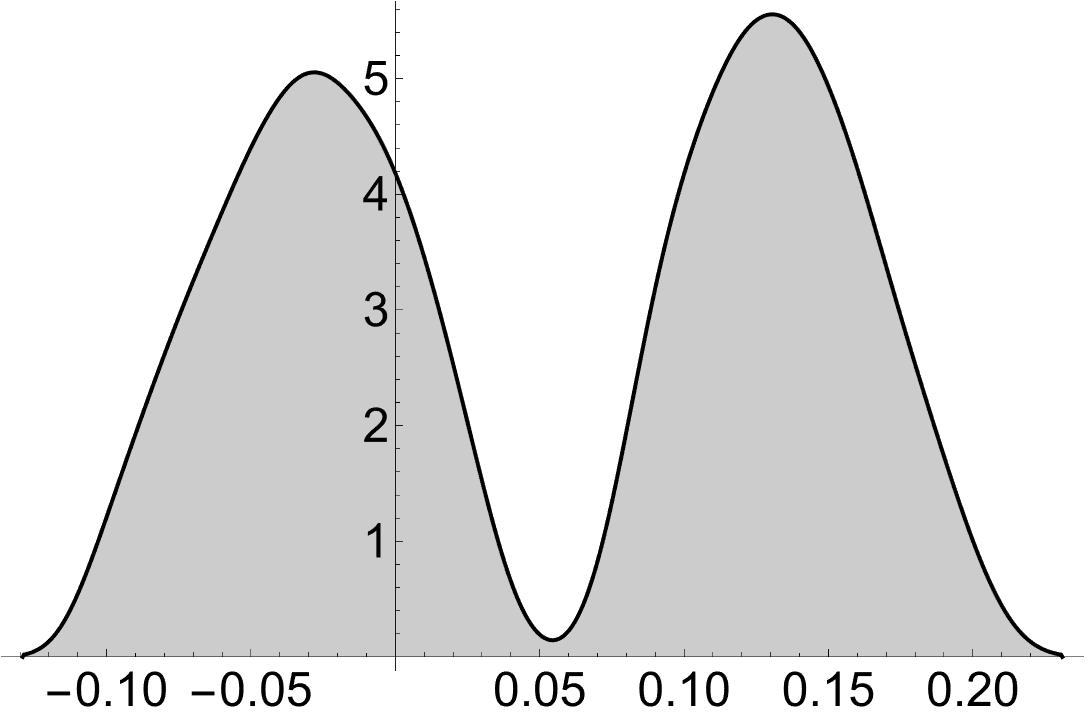}};
\node at (3, 0.5) {$0.2 x - 0.35y + 0.2z$};

    \end{tikzpicture}
    \caption{Left: the point cloud $(E_n, E_{n+1}, E_{n+2})$. Right: a projection onto a one-dimensional line.}
    \label{fig:numerics}
    \end{figure}
\end{center}

\textbf{Lyapunov functions.} We propose a greedy method to add $x_{n+1}$ to an existing set of points $x_1, \dots, x_n$. One way to control such processes is to find a Lyapunov function that can help control the dynamical behavior. An obvious candidate is the energy itself since that is what is being minimized. It would be helpful if there was a statement of the form `if $E(x_1, \dots, x_n)$ is very large, then $E(x_1, \dots, x_{n+1})$ is smaller'. Unfortunately, this is false (albeit usually true).

\begin{proposition}
    There exists a set of $n$ points such that
    $$ E(x_1, \dots, x_n) \geq \frac{n}{100}$$
    as well as, for all $0 \leq x \leq 1$,
    $$ E(x_1, \dots, x_{n},x) \geq E(x_1, \dots, x_n) + \frac{1}{10}.$$
\end{proposition}

The example is completely explicit. This should not be misconstrued to mean that the method fails. Indeed, in this particular example, we also have
$$ E(x_1, \dots, x_{n+3}) < E(x_1, \dots, x_n) - \frac{1}{10}$$
and the method performs quite beautifully after that, it does seem to be a purely local obstruction. However, it shows that the energy landscape is at least somewhat intricate and non-convex. This limits what one can hope to achieve with arguments that only consider the choice of the next point, $x_{n+1}$. Lyapunov functions may exist, we have been unable to find any. The continuous setting is much better behaved and more can be said, see \S 1.6.

\subsection{A continuous limit}
The difficulties in rigorously understanding the process seem to be rooted in the fact that the energy can be as small as $\sim 1$. Adding a single point can lead to global changes and may change the energy also at scale $\sim 1$. Adding a single point, even if there are already \textit{many} points present, may not be a `small' change. As is often the case, if a process is extremely adaptive to minuscule changes in the local structure, this makes the process more difficult to analyze. The purpose of this section is to introduce a continuous limit of the process that behaves `essentially' like the discrete process, can be analyzed and may, at least partially, explain why the discrete process works so well.\\

\textbf{The setting.} Let $\mu$ be a probability measure on $[0,1]$. For ease of exposition, we assume that it is absolutely continuous with respect to the Lebesgue measure and that $\phi(x) = [d\mu/dx]$ has a smooth density that is bounded away from 0 and from above. If $\mu= \phi(x) dx$ and if we denote its cumulative distribution function by
$$ \Phi(x) = \int_{0}^{x} \phi(y) dy,$$
then the continuous analogue of our energy (see \S 4 for a derivation) is
$$ E(\mu) = \int_0^1 \left| \Phi^{-1}(x) -x \right|dx = \int_0^1 \left| x - \Phi(x) \right| \phi(x)dx.$$
Clearly, $E(\mu) \geq 0$ with equality if and only if $\mu$ is the uniform distribution.
One could consider the behavior of the functional under general perturbations of the measure $\mu \rightarrow \mu + \varepsilon \nu$, however, this would not accurately represent the discrete process: in the discrete process we add a single point at a time.  To account for that, we revisit the discrete setting where we have an empirical measure
$$ \mu_n = \frac{1}{n} \sum_{i=1}^{n} \delta_{x_i} \qquad \mbox{and replace it by} \qquad  \mu_{n+1} = \frac{1}{n+1} \left(\sum_{i=1}^{n} \delta_{x_i} + \delta_{x_{n+1}}\right)$$
in the next step. The new measure can be rewritten as
$$ \mu_{n+1} = \left(1-\frac{1}{n+1}\right) \mu_n ~+~ \frac{1}{n+1} \delta_{x_{n+1}}.$$
This suggests the appropriate type of perturbation: rescaling the original measure and accounting for the loss of mass in the direction of a Dirac measure. Formally, given the probability measure $\mu = \phi(x) dx$ on $[0,1]$ and a point $0 \leq x \leq 1$, we define the change of energy under adding an infinitesimal amount of mass in $x$ as
$$ \frac{\partial E(\mu)}{\partial x} := \lim_{\varepsilon \rightarrow 0^+} \frac{1}{\varepsilon} \left( E( (1-\varepsilon) \mu + \varepsilon \delta_x) - E(\mu)\right).$$
This formal derivative $\partial E(\mu)/\partial x$ can be computed exactly and has some \textit{very} interesting properties. For example, the derivative can only be minimal in fixed points $\Phi(x) = x$ (compare with Theorem 1). These properties make it possible for the continuous functional to have all the structural properties that we were able to establish in the discrete setting (see Theorem 3).
Most importantly, the continuous limit does not allow for spurious local minima. Moreover, if one is far from the uniform distribution (in two concrete ways), then there is an upper bound on the gradient and more rapid improvement.
\begin{theorem}
    Let $\mu = \phi(x) dx$ be a probability measure on $[0,1]$ with $0 < c < \phi(x) < C < \infty$ such that $x = \Phi(x)$ has only finitely many solutions.
    Then
    $$ \min_{0 \leq x \leq 1}  \frac{\partial E(\mu)}{\partial x} < 0.$$
    The global minimum of $\partial E(\mu)/\partial x$ is attained in a fixed point of the cumulative distribution function, it satisfies $\Phi(x) = x$. If $\phi$ is `far from the uniform distribution', then the gradient is `very' negative in a quantitative sense:
    \begin{enumerate}
        \item    for every interval $J \subset [0,1]$ where either $\Phi(x) > x$ or $\Phi(x) < x$, 
    $$ \min_{0 \leq x \leq 1}  \frac{\partial E(\mu)}{\partial x} \leq - \frac{|J|^2}{4}$$
    \item and if the equation $\Phi(x) = x$ has only $S$ solutions, then
        $$ \min_{0 \leq x \leq 1}  \frac{\partial E(\mu)}{\partial x} \leq - \frac{1}{4S}.$$
    \end{enumerate}
\end{theorem}

The result shows an absence of spurious local minima. Moreover, the best way to decrease the energy is to add additional mass into one of the solutions of $\Phi(x) = x$. We note that this equation always has at least two solutions, $0$ and $1$. Furthermore, if there is a large gap between two consecutive solutions of $\Phi(x) = x$, then the gradient is quite negative and dramatic improvement is possible; the same is true if there are only few solutions of $\Phi(x) = x$.

\subsection{Related results} The special case where we start with an empty set of points and try to place a sequence of points in $[0,1]$ so that the first $n$ points approximate the uniform distribution uniformly in $n$ is an old problem. It has a well-developed theory, we refer to the classic books by Beck and Chen \cite{beckchen}, Chazelle \cite{chaz}, Dick and Pillichshammer \cite{dick}, Drmota and Tichy \cite{drmota} and Kuipers and Niederreiter \cite{kuipers}. A very rough summary of results would be
\begin{enumerate}
    \item The situation is somewhat understood in $[0,1]$ with regards to the leading order of magnitude of many `error' quantities \cite{tat, corput1, corput2, graham, schm}.
    \item The problem of finding the `best' example with regards to a natural metric is wide open even in $d=1$ dimension \cite{puch, ost}
    \item and there are few examples of `nice' sequences in $[0,1]$ whose construction usually relies on either combinatorial or number-theoretic principles.
    \item The situation is not understood in $[0,1]^d$ when $d \geq 2$ \cite{bilyk, roth}
    \item and, when $d \geq 2$, it is not even clear what to conjecture \cite{bilyk0}.
\end{enumerate}
Our work grew out of a desire to understand whether constructions based on either greedy or potential-theoretic `energy-based' methods could be effective in constructing new types of such sequences in $[0,1]$ with the ultimate goal being new such examples in $[0,1]^d$. Related results in this direction are 
\cite{brown, brown2, clement, clement2, clement3, kritzinger, pausinger, stein4, stein3, stein2, stein5}. We are not aware of this type of question (starting with a given initial set, updating in a way to be close to a given target measure) having been studied anywhere. If we start with an empty set and know in advance how many points to add, this leads us to the question of quantization of probability measures which is very well studied \cite{graf}. If the number of points is not known, this leads to questions of uniform distribution and the results cited above. A greedy variant of quantization was studied by Luschgy-Pag\`es \cite{lusch} (see also \cite{lusch2}). Some related ideas are presented in \cite{gomez, pron, pron2}.

\section{Proof of Theorem 1}
\subsection{Balancing property.} We first show that the new point is always a rational number of the form $x_{n+1} = k/(n+1)$ with $1 \leq k \leq n$.
\begin{proof}  We may assume that the first $n$ points are ordered $0 \leq x_1 \leq x_2 \leq \dots \leq x_n \leq 1$. Suppose we now add a point so as to minimize the energy for $n+1$ points. If the new point being added satisfies $x \leq x_1$, then the energy is
$$ E(x,x_1, \dots, x_n) = \left|x - \frac{1}{n+1} \right| + \sum_{i=1}^{n} \left|x_i - \frac{i+1}{n+1} \right|.$$
If $1/(n+1) \leq x_1$, then the obviously optimal choice is $x = 1/(n+1)$. We will now argue that in the case $1/(n+1) > x_1$, adding the new point in $[0,x_1]$ cannot be optimal. For this sake, let us define $\ell \in \mathbb{N}_{\geq 0}$ so that $x_1 = \dots = x_{1+\ell} < x_{2 + \ell}$ keeps track of how many points are located in $x_1$. If $1/(n+1) > x_1$ and we add the new point in $[0,x_1]$, then the energy is minimized by making $x$ as large as the restriction allows and we get
\begin{align*}
     E_1(x,x_1, \dots, x_n) &= \left|x - \frac{1}{n+1} \right| + \sum_{i=1}^{n} \left|x_i - \frac{i+1}{n+1} \right| \\
     &= \left|x_1 - \frac{1}{n+1} \right| + \sum_{i=1}^{1+\ell} \left|x_i - \frac{i+1}{n+1} \right| + \sum_{i=2+\ell}^{n} \left|x_i - \frac{i+1}{n+1} \right|\\
     &=  \left|x_1 - \frac{1}{n+1} \right| + \sum_{i=1}^{1+\ell} \left|x_1 - \frac{i+1}{n+1} \right| + \sum_{i=2+\ell}^{n} \left|x_i - \frac{i+1}{n+1} \right|.
\end{align*}
We will compare this energy to what happens if we add $x$ in the interval $(x_{1+\ell}, x_{2+\ell})$. This would make $x$ the $(2+\ell)-$th largest point and the energy becomes
\begin{align*}
     E_2(x_1, \dots, x, \dots, x_n) &=\sum_{i=1}^{1+\ell} \left|x_1 - \frac{i}{n+1} \right| + \left|x - \frac{\ell+2}{n+1} \right| + \sum_{i=2+\ell}^{n} \left|x_i - \frac{i+1}{n+1} \right|.
\end{align*}
We first keep $x \in (x_{1+\ell}, x_{2+\ell})$ arbitrary and will then show that $x = x_{1 + \ell} + \varepsilon$ for $0 < \varepsilon \ll 1$ sufficiently small is an admissible choice that ensures that $E_2 < E_1$ which is a contradiction. The difference between these two energies is
\begin{align*}
    E_1 - E_2 &=  \left|x_1 - \frac{1}{n+1} \right| -  \left|x - \frac{\ell+2}{n+1} \right| + \sum_{i=1}^{1+\ell} \left( \left|x_1 - \frac{i+1}{n+1} \right| - \left|x_1 - \frac{i}{n+1} \right| \right).
\end{align*}
By assumption, we have $x_1 < 1/(n+1)$, therefore
\begin{align*}
    \sum_{i=1}^{1+\ell} \left( \left|x_1 - \frac{i+1}{n+1} \right| - \left|x_1 - \frac{i}{n+1} \right| \right) &= \sum_{i=1}^{1+\ell} \left(  \frac{i+1}{n+1} - x_1\right) -  \left(\frac{i}{n+1} - x_1\right) \\
    &= \sum_{i=1}^{1+\ell}  \frac{i+1}{n+1} - \frac{i}{n+1}  =  \frac{1+\ell}{n+1} .
\end{align*}
Therefore, using the triangle inequality,
\begin{align*}
    E_1 - E_2 &= \frac{1+\ell}{n+1} + \left|x_1 - \frac{1}{n+1} \right| -  \left|x - \frac{\ell+2}{n+1} \right| \\
    &\geq \frac{1+\ell}{n+1} + \left|x_1 - \frac{1}{n+1} \right| -  \left|x - \frac{1}{n+1} \right| -\frac{1+\ell}{n+1} \\
    &= \left|x_1 - \frac{1}{n+1} \right| -  \left|x - \frac{1}{n+1} \right|
\end{align*} 
However, since $x_1 < 1/(n+1)$ and $x > x_1$ is an admissible choice, we see that we can make the expression strictly positive and $E_2 < E_1$. This concludes this case. If the new point being added is larger than all the existing points $x_{n} \leq x$, then
$$ E(x_1, \dots, x_n, x) = \sum_{i=1}^{n} \left| x_i - \frac{i}{n+1} \right| + \left|x - 1 \right|$$
which is now clearly minimized if we choose $x=1$ (which is guaranteed to not violate the condition $x_n \leq x$). It remains to deal with the case where the new point being added satisfies $x_k \leq x < x_{k+1}$. The energy is then given by
$$ E(x_1, \dots, x, \dots, x_n) = \sum_{i=1}^{k} \left|x_i - \frac{i}{n+1} \right| + \left| x - \frac{k+1}{n+1} \right| + \sum_{i=k+1}^{n} \left|x_i  - \frac{i+1}{n+1} \right|.$$
If
$$ x_k \leq \frac{k+1}{n+1} \leq x_{k+1},$$
then we are done and $x = (k+1)/(n+1)$ is the optimal choice. Suppose now that
$ (k+1)/(n+1) < x_k.$
If that is the case, then the best choice for $x$ subject to the constraints is $x = x_k$. The same computation as above shows that choosing $x = x_{k} - \varepsilon$ for $\varepsilon$ sufficiently small leads to a smaller energy. The case $ x_{k+1} \leq (k+1)/(n+1)$
is completely analogous. This shows that $x_{n+1}$ is always a rational number with denominator $n+1$. It also shows that if $x_{n+1} = (k+1)/(n+1)$, then
$$ x_k \leq \frac{k+1}{n+1} \leq x_{k+1}$$
which was the desired result.
\end{proof}

\subsection{Fast computation.}
Assume that the first $n$ points are ordered $0 \leq x_1 \leq x_2 \leq \dots \leq x_n \leq 1$.  We already know that the next point being added is going to be of the form $k/(n+1)$ for some $1 \leq k \leq n$. Supposing the points are sorted $x_1 \leq x_2 \leq \dots \leq x_n$, it suffices to start the computation from an extremal interval (either $0\leq x \leq x_1$ or $x_n \leq x \leq 1$), evaluate the sum, then update \emph{only} the sum-term that changes. Indeed, the next point added does not contribute to the new energy, we only need to track the contributions of previous points.
If $0 \leq x_{n+1} \leq x_1$, we know that $x_{n+1} = 1/(n+1)$ which means that the energy would then be
$$ \left|x_{n+1} - \frac{1}{n+1} \right| + \sum_{i=1}^n \left|x_i-\frac{i+1}{n+1}\right| = \sum_{i=1}^n \left|x_i-\frac{i+1}{n+1}\right|.$$ 
Optimizing for the best choice $x_1 \leq x_{n+1} \leq x_2$, the change in the energy sum is that the term $\left|x_1-2/(n+1)\right|$ is replaced by $\left|x_1-1/(n+1)\right|$. Only this term needs to be updated to obtain the new energy, this can be done in $\mathcal{O}(1)$ time. This process can be repeated for all the intervals: when moving $x$ to the $x_i\leq x \leq x_{i+1}$ interval from the $x_{i-1}\leq x \leq x_i$ interval, we replace $\left|x_i-(i+1)/(n+1)\right|$ by $\left|x_i-i/(n+1)\right|$.
The initial sum requires $\mathcal{O}(n)$ operations, while the $n$ different updates require $\mathcal{O}(1)$ other operations each. Overall, only a linear number of operations are required to compute the next element in the sequence.

\begin{algorithm}[t]
\caption{Fast computation of the sequence}
\begin{algorithmic}[1]
\Require Starting set $L$, of initial size $l$, target number of points $n$
\State Sort L
\For{$i=l$ to $i=n-1$} 
    \State Set $s=\sum_{k=1}^{i}|L[k]-\frac{k+1}{i+1}|$, the initial sum if the point is added in $[0,L[1]]$.
    \State Define $c=s$, current lowest sum value found.
    \State Define $p=0$, position where we found the current lowest sum $c$.
    \For{$j=1$ to $j=i$}
        \State We now consider $x \in [L[j];L[j+1]]$, with $L[i+1]=1$.
        \State Update the sum value $s=s-|L[j]-\frac{j+1}{i+1}|+|L[j]-\frac{j}{i+1}|$
        \If{$s<c$}
        \State Update the current best sums and associated position: $c=s$ and $p=j$.
        \EndIf
    \EndFor
    \State Add the new point $x=\frac{p+1}{i+1}$ to $L$: $L=L[:p+1]+[x]+L[p+1:]$
    \EndFor
    \State Return $L$
\end{algorithmic}
\end{algorithm}

\subsection{Heuristic derivation of the behavior} We now derive the heuristic
 $$ x_{n+1} \quad \sim \quad \arg\max_{0 \leq x \leq 1}  \quad \int_0^1 \sign\left[\Delta(y)\right] \cdot h_x(y) dy.$$
If we add a new point in $0 \leq x \leq 1$, then the change in energy can be written as
\begin{align*}
  \Delta E &= \sum_{x_i < x} \left( \left| x_i - \frac{i}{n+1} \right| -  \left| x_i - \frac{i}{n} \right| \right)  +  \sum_{x_i > x} \left( \left| x_i - \frac{i+1}{n+1} \right| -  \left| x_i - \frac{i}{n} \right| \right).
\end{align*} 
A useful first order asymptotic is
\begin{align*}
     \left| x_i - \frac{i}{n+1} \right| -  \left| x_i - \frac{i}{n} \right| &\sim \frac{1}{n}\sign\left(x_i - \frac{i}{n} \right) \cdot \frac{i}{n} \\
       \left| x_i - \frac{i+1}{n+1} \right| -  \left| x_i - \frac{i}{n} \right| &\sim \frac{1}{n}\sign\left(x_i - \frac{i}{n} \right) \cdot \frac{i - n}{n}.
\end{align*}
This suggests introducing the function $h_x:[0,1] \rightarrow [0,1]$ defined by
$$ h_x(y) = \begin{cases} y \qquad &\mbox{if}~y \leq x \\ 
y-1 \qquad &\mbox{if}~y > x. \end{cases}$$
This allows to write the change in energy, at least approximately, as
$$ \Delta E \sim  \frac{1}{n}\sum_{x_i < x}  \sign\left(x_i - \frac{i}{n} \right)  h_x(i/n) + \frac{1}{n}\sum_{x_i > x} \sign\left(x_i - \frac{i}{n} \right)  h_x(i/n). $$
Given $0 \leq x_1, \dots, x_n \leq 1$, we introduce the discrepancy function
$$ \Delta(x) = \frac{\# \left\{1 \leq i \leq n: x_i \leq x \right\}}{n} - x$$
and observe that $\Delta(x) > 0$ indicates that there are `too many' points in $[0,x]$ which shows that, locally, $x_i < i/n$, whereas $\Delta(x) < 0$ indicates that there are `not enough' points in $[0,x]$ which is indicative of $x_i > i/n$. Therefore, we expect
$$ \sign \left[\left(x_i - \frac{i}{n}\right)\right] = -\sign \left[ \Delta(x_i) \right].$$
Thus minimizing $\Delta E$ is roughly the same as maximizing
$$-\Delta E \sim \frac{1}{n}\sum_{x_i < x}  \sign \left[ \Delta(x_i) \right] h_x(i/n) + \frac{1}{n}\sum_{x_i > x} \sign \left[ \Delta(x_i) \right]  h_x(i/n)$$
and assuming that the points are approximately uniformly distributed, we may approximate the quantity of interest a little bit further and write
$$ x_{n+1} \quad \sim \quad \arg\max_{0 \leq x \leq 1} \int_0^1 \sign\left[\Delta(y)\right] \cdot h_x(y) dy.$$

\section{Proofs}

\subsection{Proof of Proposition 1.}
\begin{proof} 
    We know $x_{n+1}=k/(n+1)$ for some $k \in \{1,\ldots,n+1\}$ and its contribution to the energy is 0. Therefore
    \begin{align*}
    E(x_1, \dots, x_n) + E(x_1, \dots, x_{n+1}) &= \sum_{1 \leq i \leq n \atop x_i \leq x_{n+1}} \left| x_i - \frac{i}{n} \right| + \left| x_i - \frac{i}{n+1} \right|  \\
    &+ \sum_{1 \leq i \leq n \atop x_i > x_{n+1}}  \left| x_i - \frac{i}{n} \right| + \left| x_i - \frac{i+1}{n+1} \right|.
    \end{align*}
We have, for all $x \in \mathbb{R}$,
$$ \left| x - \frac{i}{n} \right| + \left| x - \frac{i}{n+1} \right| \geq \frac{i}{n(n+1)}$$
as well as
$$  \left| x_i - \frac{i}{n} \right| + \left| x_i - \frac{i+1}{n+1} \right| \geq \frac{n-i}{n(n+1)}.$$
Therefore
$$ E(x_1, \dots, x_n) + E(x_1, \dots, x_{n+1}) \geq \sum_{i=1}^n \min\left(\frac{i}{n(n+1)},\frac{n-i}{n(n+1)}\right).$$
A short computation then implies the result.
\end{proof}
There is an alternative argument that implies slightly larger constants when taking averages over more consecutive elements. Consider an individual point, say $0 < x< 1$. This point can be of the form $i/n$ and it could be of the form $j/(n+1)$ but it is unlikely to be both simultaneously. Using $\| \cdot \|$ to denote the distance to the nearest integer, we have that
$$ \min_{k \in \mathbb{N}} \left| x - \frac{k}{n} \right| = \frac{\| nx \|}{n}.$$
Then,
$$ \sum_{r = n}^{n+ \ell} \min_{k \in \mathbb{N}} \left| x - \frac{k}{r} \right| = \sum_{r = n}^{n+ \ell}  \frac{\| rx \|}{r} \geq \frac{1}{n+\ell}  \sum_{r = n}^{n+ \ell} \| rx \|.$$
The sum $\sum_{r = n}^{n+ \ell} \| rx \|$ is interesting in itself. It evaluates to 0 when $x=0,1$ and, by continuity, for values close to those two values. However, if $ \sum_{i=1}^{n} \left|x_i - i/n \right| \lesssim 1$, then `most' elements are far away from 0 and 1. It is not clear to us whether the sum has been studied before (the distribution of the summands has been considered by Don \cite{don}). This can be used to prove that averages over more terms have an average larger than $1/8$, however, we have not pursued this further.

\begin{center}
    \begin{figure}[h!]
    \begin{tikzpicture}
\node at (0,0) {\includegraphics[width=0.4\textwidth]{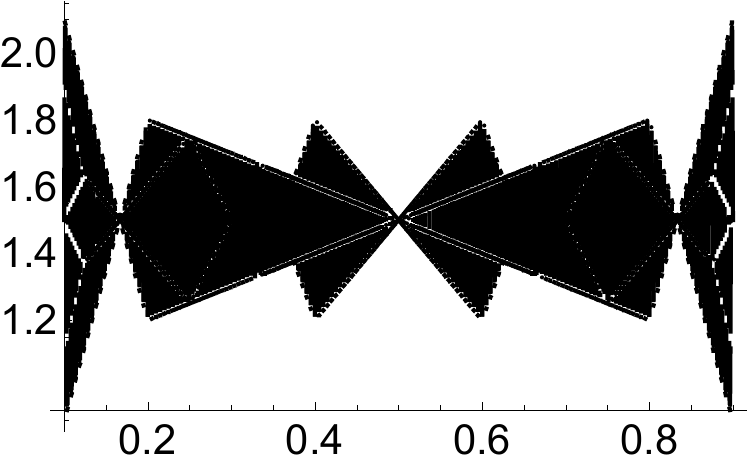}};
\node at (0.4, 1.2) {$\sum_{n=200}^{205} \|nx\|$};
\node at (6,0) {\includegraphics[width=0.4\textwidth]{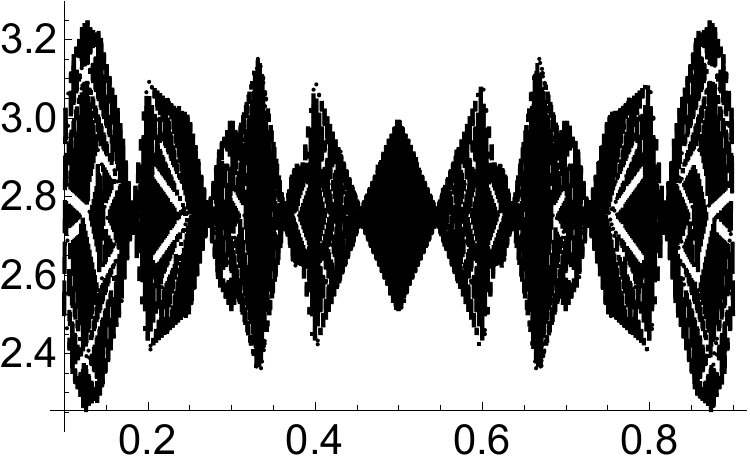}};
\node at (6.4, 1.4) {$\sum_{n=300}^{310} \|nx\|$};
    \end{tikzpicture}
       \end{figure}
\end{center}

\subsection{Proof of Theorem 2.}
\begin{proof}
The notion of energy is naturally related to Optimal Transport in the Wasserstein $W_1$ metric.
This quantity, in turn, has a simple formula in terms of the cumulative distribution function (see below).
   Since the original measure is discrete, there is an explicit formula for the $W_1-$distance: assuming without loss of generality that $0 \leq x_1 \leq \dots \leq x_n \leq 1$, we have
    $$  W_1\left(   \frac{1}{n} \sum_{k=1}^{n} \delta_{x_k}, 1_{[0,1]}(x) dx\right) = \sum_{i=1}^{n} \int_{(i-1)/n}^{i/n} \left| x_i - y \right| dy.$$
    Taking the lower bound
    $$ \int_{(i-1)/n}^{i/n} \left| x_i - y \right| dy \geq \frac{1}{n} \left|x_i - \frac{i-0.5}{n} \right| \geq \frac{1}{n} \left|x_i - \frac{i}{n} \right| - \frac{1}{n^2}$$
    and summing over all $n$ terms implies
    $$   W_1\left(   \frac{1}{n} \sum_{k=1}^{n} \delta_{x_k}, 1_{[0,1]}(x) dx\right) \geq \frac{E(x_1, \dots, x_n)}{n} - \frac{1}{n}.$$
    Simultaneously, we have the upper bound
    $$ \int_{(i-1)/n}^{i/n} \left| x_i - y \right| dy \leq \frac{1}{n}\left|x_i - \frac{i}{n} \right| + \frac{1}{n^2}$$
    from which we deduce
    $$   W_1\left(   \frac{1}{n} \sum_{k=1}^{n} \delta_{x_k}, 1_{[0,1]}(x) dx\right) \leq \frac{E(x_1, \dots, x_n)}{n} + \frac{1}{n}$$
 and thus
 $$ \left| n \cdot  W_1\left(   \frac{1}{n} \sum_{k=1}^{n} \delta_{x_k}, 1_{[0,1]}(x) dx\right) - E(x_1, \dots, x_n) \right| \leq 1.$$
  A standard result in optimal transport in one dimension is an exact formula for the $W_1-$cost (see \cite[Theorem 2.18]{villani}): if
 $ F(x) = \int_{-\infty}^{x} d\mu$
 is the CDF of a probability measure $\mu$ supported on $[0,1]$, then
 \begin{align*}
     W_1(\mu, dx) = \int_0^1 \left|F(x) - x \right| dx.
 \end{align*}
 Therefore
 $$  W_1\left(   \frac{1}{n} \sum_{k=1}^{n} \delta_{x_k}, 1_{[0,1]}(x) dx\right) = \int_0^1 \left| \frac{\# \left\{1 \leq i \leq n: x_i \leq x \right\}}{n} - x\right| dx.$$
 This quantity is also known as the $L^1-$star discrepancy. At this point, we invoke a powerful result of Hal\'asz \cite{halasz}. It says that if $P \subset [0,1]^2$ is a set of $\# \mathcal{P}$ points $\mathcal{P} = \left\{(x_1, y_1), \dots, (x_n, y_n) \right\}$, then
 $$ \int_0^1 \int_0^1 \left| \frac{\# \left\{1 \leq i \leq \# \mathcal{P}: x_i \leq x \wedge y_i \leq y\right\}}{\# \mathcal{P}} - xy \right| dx dy \geq c \sqrt{\log{(\# \mathcal{P})}}.$$
 It is well understood that the problem of understanding sequences in $[0,1]$ is `equivalent' to understanding the behavior of sets in $[0,1]^2$ by using a standard construction. The details are as follows. The goal is to show that for every one-dimensional sequence $(x_k)_{k=1}^{\infty}$ there are infinitely many values $n \in \mathbb{N}$ such that
$$ \int_0^1 \left| \frac{\# \left\{1 \leq i \leq n: x_i \leq x \right\}}{n} - x\right| dx \geq c \sqrt{\log{n}}.$$
Suppose this is false. Then, for every $\varepsilon > 0$ we can find a sequence $(x_k)_{k=1}^{\infty}$ such that the $L^1-$discrepancy is $\leq \varepsilon \sqrt{\log{n}}$ for all but finitely many terms.  
 Given a one-dimensional sequence of points $x_1, \dots, x_n$ in $[0,1]$, we can turn them into a two-dimensional set of points in $[0,1]^2$ by writing
 $$ \mathcal{P} = \left\{ (x_1, \tfrac{1}{n}), (x_2, \tfrac{2}{n}), \dots, (x_n, \tfrac{n}{n})\right\}.$$
 We assume that $n$ is extremely large, much larger than all the exceptional values and so large that the exceptional values contribute an arbitrarily small amount in the subsequent computation.
 The number of elements of $\mathcal{P}$ that can be found in $[0,x] \times [0,y]$ is really the number of elements of the first $y \#\mathcal{P}$ elements of the sequence that are in $[0,x]$, formally
\begin{align*}
    \frac{\# \left\{1 \leq i \leq \# \mathcal{P}: x_i \leq x \wedge y_i \leq y\right\}}{\# \mathcal{P}} &=  \frac{\# \left\{1 \leq i \leq \left\lfloor y \# \mathcal{P} \right\rfloor: x_i \leq x  \right\}}{\# \mathcal{P}} \\
    &=  \frac{\# \left\{1 \leq i \leq \left\lfloor y \# \mathcal{P} \right\rfloor: x_i \leq x  \right\}}{\left\lfloor y \# \mathcal{P} \right\rfloor} \frac{\left\lfloor y \# \mathcal{P} \right\rfloor}{\# \mathcal{P}} \\
    &= \frac{\# \left\{1 \leq i \leq \left\lfloor y \# \mathcal{P} \right\rfloor: x_i \leq x  \right\}}{\left\lfloor y \# \mathcal{P} \right\rfloor} \frac{ y \# \mathcal{P} + \mathcal{O}(1)}{\# \mathcal{P}} \\
    &=  \frac{\# \left\{1 \leq i \leq \left\lfloor y \# \mathcal{P} \right\rfloor: x_i \leq x  \right\}}{\left\lfloor y \# \mathcal{P} \right\rfloor} y + \mathcal{O}( \# P^{-1}).
\end{align*}  
Then the integral over $x$ satisfies
\begin{align*}
 I &= \int_0^1 \left|  \frac{\# \left\{1 \leq i \leq \# \mathcal{P}: x_i \leq x \wedge y_i \leq y\right\}}{\# \mathcal{P}} - xy \right| dx \\
 &= \int_0^1 \left| \frac{\# \left\{1 \leq i \leq \left\lfloor y \# \mathcal{P} \right\rfloor: x_i \leq x  \right\}}{\left\lfloor y \# \mathcal{P} \right\rfloor} y - xy \right| dx +  \mathcal{O}( \# \mathcal{P}^{-1}) \\
 &= y \int_0^1 \left| \frac{\# \left\{1 \leq i \leq \left\lfloor y \# \mathcal{P} \right\rfloor: x_i \leq x  \right\}}{\left\lfloor y \# \mathcal{P} \right\rfloor}  - x \right| dx +  \mathcal{O}( \# \mathcal{P}^{-1}).
\end{align*}
Now, for our infinite hypothetical sequence, we can integrate this over $y$ and get a bound that is smaller than $\leq \varepsilon \sqrt{\log{\# \mathcal{P}}}$ which contradicts Hal\'asz' result and the argument follows.
\end{proof}

The argument generalizes to other $W_p$ distances, we refer to Graham \cite{graham}, and it is philosophically connected to a result of Proinov \cite{proinov} (see also Kirk \cite{kirk}).

\subsection{Proof of Proposition 3.}
\begin{proof}
    Proposition 3 follows from the following explicit example.  Let $n = 4m$ be a multiple of 4 and set
    $$ x_k = \begin{cases}
        0 \qquad &\mbox{if}~k \leq m \\
       k/n\qquad &\mbox{if}~m < k \leq 3m \\
       1 \qquad &\mbox{if}~k > 3m.
    \end{cases}$$
    Some fairly standard computations show that 
    $$ E(x_1, \dots, x_n) =  \frac{n + o(n)}{8}$$
    and that the next point $x_{n+1}$ is being added in the middle, $x_{n+1} = (2m)/(4m+1)$ and that the energy increases by, asymptotically, $1/8$. Some more work shows that $x_{n+2} \sim 3/4$ (with energy comparable to $E(x_1, \dots, x_{n+1})$ and that $x_{n+3} \sim 1/4$ with the energy satisfying
    $$ E(x_1, \dots, x_{n+3}) \sim E(x_1, \dots, x_n) - \frac{1+o(1)}{12}.$$
\end{proof}

\vspace{-10pt}
\begin{center}
    \begin{figure}[h!]
\begin{tikzpicture}
\node at (0,0) {\includegraphics[width=0.4\textwidth]{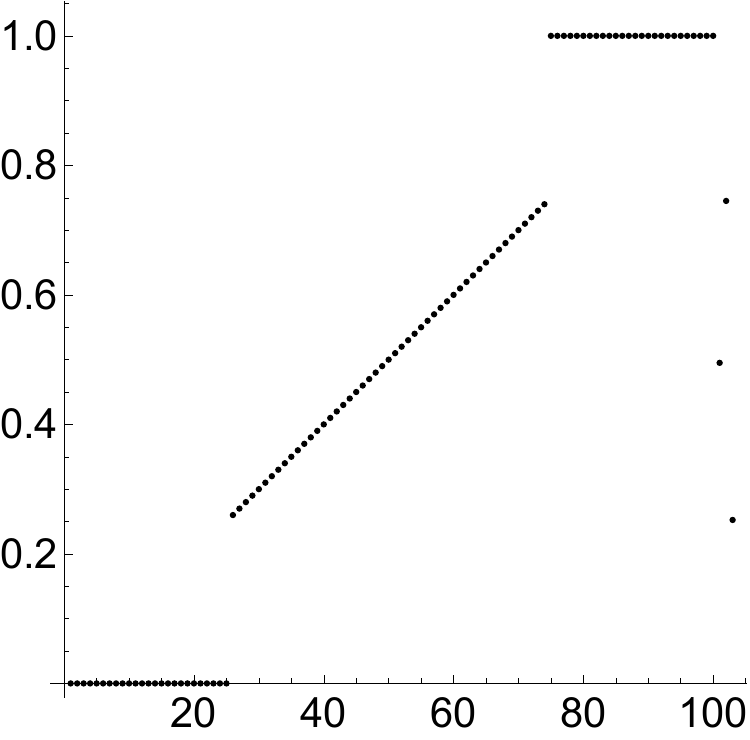}};
\node at (6,0) {\includegraphics[width=0.4\textwidth]{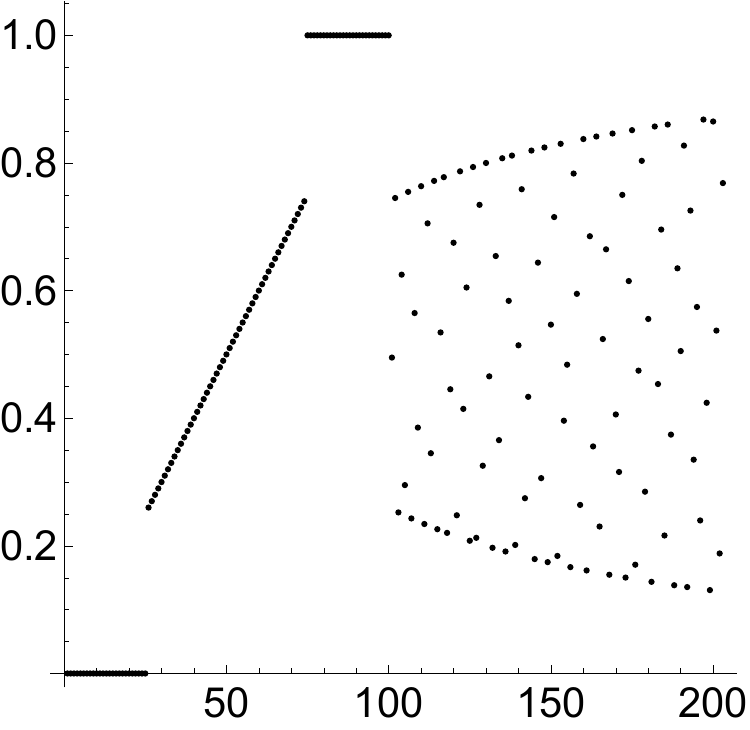}};
\end{tikzpicture}
\caption{Example from Proposition 3 when $m=25$.}
    \end{figure}
\end{center}

\section{The Continuous Limit}
\subsection{Deriving the functional.}
Suppose that $\mu$ is a probability measure on $[0,1]$ that admits a smooth density 
$ \mu = \phi(x) dx$
and that this density is bounded away from 0. We also introduce the cumulative distribution function 
$$ \Phi(x) = \cdf[\mu](x) = \int_0^x d\mu = \int_0^x \phi(y) dy.$$
If we were to sample $n \gg 1$ elements from that distribution, then the $i-$th largest element would, as $n \rightarrow \infty$, be located near $x = \Phi^{-1}(i/n)$. The theory of order statistics suggests the approximation
$$ \sum_{i=1}^{n} \left|x_i - \frac{i}{n}\right| \sim  \sum_{i=1}^{n} \left| \Phi^{-1}\left( \frac{i}{n} \right) - \frac{i}{n}\right|.$$
Taking a continuous limit, we arrive at
$$ \lim_{n \rightarrow \infty} \frac{1}{n}\sum_{i=1}^{n} \left| \Phi^{-1}\left( \frac{i}{n} \right) - \frac{i}{n}\right|  = \int_0^1 \left| \Phi^{-1}(t) - t \right|dt.$$
We use that $\Phi:[0,1] \rightarrow [0,1]$ is a smooth bijective map allowing us to change variables to obtain the alternative representation
$$ \int_0^1 \left| \Phi^{-1}(t) - t \right|dt = \int_0^1 \left| t - \Phi(t) \right| \cdot \phi(t) dt.$$

\subsection{The derivative of the functional.}
In the discrete setting, we are given a set of $n$ points $x_1, \dots, x_n$ and the corresponding
probability measure 
$$ \mu_1 = \frac{1}{n} \sum_{i=1}^{n} \delta_{x_i}.$$
The new point $x_{n+1}$ is being added so as to minimize the energy functional and leads to a new measure $\mu_2$ that can be rewritten as
$$ \mu_2 = \frac{1}{n+1} \left( \delta_{x_{n+1}} + \sum_{i=1}^{n} \delta_{x_i} \right) = \frac{n}{n+1} \mu_1 + \frac{\delta_{x_{n+1}}}{n+1}.$$
In the continuous limit, the natural scaling is therefore to replace $\mu$ by
$$ \mu \rightarrow \mu_{\varepsilon} = (1-\varepsilon) \mu + \varepsilon \delta_x$$
where $0<x<1$ is chosen so that the energy functional is as small as possible and we should think of $\varepsilon>0$ as infinitesimally small. In terms of order of quantifiers, we think of the optimal $x$ as a function of $x = x(\varepsilon)$ and let $\varepsilon \rightarrow 0$. We note that, outside of the new point $x$, the probability density function of $\mu_{\varepsilon}$ is given by $(1-\varepsilon)\phi(x)$.
The probability density function of $\mu_{\varepsilon}$ is not formally defined in $x$, however, it is defined in the usual weak sense as Dirac measure $\varepsilon \delta_x$. The change in the cumulative distribution function is also completely explicit and 
$$ \cdf[ (1-\varepsilon) \mu + \varepsilon \delta_x](y) =  (1-\varepsilon) \Phi(y) + \varepsilon\cdot 1_{y>x}.
$$

To analyze the change in the energy functional
$$ E(\mu) = \int_0^1 \left| t - \Phi(t) \right| \cdot \phi(t) dt,$$
we split the integral into the regions $\left\{t < x\right\}$, $\left\{t=x\right\}$ and $\left\{t >x \right\}$. The first is
$$(\mbox{I}) =  \int_0^x |t - (1-\varepsilon) \Phi(t)| \cdot (1-\varepsilon) \phi(t) dt.$$
Using that $x - \Phi(x)$ vanishes only a finite number of times and that $\varepsilon \rightarrow 0^+$, we may think of $A = t - \Phi(t)$ as being much bigger in absolute value, $|A| \gg |B|$, than $B = \varepsilon \Phi(t)$ for `most' values of $0 \leq t \leq 1$. Using the identity $|A + B| = |A| + \sign(AB) |B|$, valid for $|A| \geq |B| \geq 0$ the term can be rewritten to leading order as
\begin{align*}
   (\mbox{I}) &= (1-\varepsilon)\int_0^x |t - (1-\varepsilon) \Phi(t)|  \phi(t) dt \\
&= (1-\varepsilon)\int_0^x |t - \Phi(t) + \varepsilon \Phi(t)|  \phi(t) dt \\
&= (1-\varepsilon)\int_0^x |t - \Phi(t)| \phi (t) dt + (1-\varepsilon) \cdot \varepsilon \int_0^x \sign(t-\Phi(t))  \Phi(t)  \phi(t) dt,
\end{align*}
The second term comes from adding $\varepsilon$ of mass in point $x$. This has very little impact
on $|t-\Phi(t)|$ around $t=x$ (there is a jump of order $\varepsilon$) but it does impact $\phi(t)$ which, suitably interpreting the Dirac measures, leads to a contribution of
$$ (\mbox{II}) = \varepsilon |x - \Phi(x)| + \mathcal{O}(\varepsilon^2).$$
The third term behaves very much like the first and leads to
\begin{align*}
  (\mbox{III}) &=  \int_x^1 |t - (1-\varepsilon) \Phi(t) - \varepsilon| (1-\varepsilon) \phi(t) dt\\
 &=  (1-\varepsilon)\int_x^1 |t - \Phi(t) + \varepsilon (\Phi(t) - 1)  | \phi(t) dt \\
&= (1-\varepsilon)\int_x^1 |t - \Phi(t)| \phi (t) dt \\
&+ (1-\varepsilon) \cdot \varepsilon \int_x^1 \sign[( t - \Phi(t))  (\Phi(t)-1)]  (1-\Phi(t))  \phi(t) dt.
\end{align*}
Note that $\Phi(t) -1$ is always negative and thus
$$ \sign[( t - \Phi(t))  (\Phi(t)-1)] (1-\Phi(t))  = \sign[t - \Phi(t)] \cdot  (\Phi(t)-1).$$

Adding all three together, we get  
\begin{align*}
    (\mbox{I})  +  (\mbox{II}) +   (\mbox{III})  &= (1-\varepsilon)  \int_0^1 |t- \Phi(t)| \phi(t) dt +  \varepsilon \int_0^x \sign( t - \Phi(t))  \Phi(t)  \phi(t) dt \\
    &+  \varepsilon \int_x^1 \sign( t - \Phi(t))  ( \Phi(t)-1)  \phi(t) dt +   \varepsilon\cdot |x - \Phi(x)| + \mathcal{O}(\varepsilon^2).
\end{align*}
Note that the first term and $\varepsilon$ should be thought of as externally given: our goal is to choose $x$ to make the remaining terms as small as possible. Thus, formally, letting $\varepsilon \rightarrow 0$, we may define the function $g:[0,1] \rightarrow \mathbb{R}$ via
\begin{align*}
\frac{\partial E(\mu)}{\partial x} &=   \int_0^x \sign(\Phi(t) - t)  \Phi(t)  \phi(t) dt \\
    &+  \int_x^1 \sign( \Phi(t) - t)  ( \Phi(t) -1)  \phi(t) dt +   |x - \Phi(x)|.
\end{align*}
$\partial E(\mu)/\partial x$ describes the minuscule change of the overall energy under adding $\varepsilon$ mass to $x$ and rescaling the rest of the measure. The goal is therefore to find the value
$x$ that minimizes $\partial E(\mu)/\partial x$. Another change of variables $z = \Phi(t)$, gives $t = \Phi^{-1}(z)$ as well as the transformation
$$ \frac{dz}{dt} = \phi(t).$$
A short computation shows that this change of variables leads to
\begin{align*}
 \frac{\partial E(\mu)}{\partial x}&=   \int_0^{\Phi^{-1}(x)} \sign(z - \Phi^{-1}(z)) z ~ dz \\
    &+  \int_{\Phi^{-1}(x)}^1 \sign(z - \Phi^{-1}(z))  (z-1) ~  dz+   |\Phi^{-1}(x) - x|.
\end{align*}
 For simplicity of expression, we perform another substitution $\Phi(\alpha)=x$ and consider the function
\begin{align*}
h(\alpha) &=  \int_0^{\alpha} \sign(z - \Phi^{-1}(z)) z~  dz  \\
&+ \int_{\alpha}^1 \sign(z - \Phi^{-1}(z)) (z- 1) ~dz + |\alpha - \Phi(\alpha)|.
\end{align*}
In particular, if $h$ attains a global minimum in $\alpha^*$, then the optimal point to add mass will be in $x^* = \Phi(\alpha^*)$.

\subsection{Simplification I}
At this point, we show that the overall algebraic structure implies that minimizers can only be attained in fixed points of $\Phi(x)$.  

\begin{lemma}
  Let $\mu = \phi(x) dx$ be a probability measure on $[0,1]$ with $0 < c < \phi(x) < C < \infty$ such that $x = \Phi(x)$ has only finitely many solutions. The minimal value of
  \begin{align*}
h(\alpha) &=  \int_0^{\alpha} \sign(z - \Phi^{-1}(z)) z~  dz  \\
&+ \int_{\alpha}^1 \sign(z - \Phi^{-1}(z)) (z- 1) ~dz + |\alpha - \Phi(\alpha)|.
\end{align*}
is attained in a point satisfying $\alpha = \Phi(\alpha)$. 
\end{lemma}
We note that since $\Phi(0) = 0$ and $\Phi(1) = 1$, such fixed points always exist.  In particular, if these are the only two fixed points of the cumulative distribution function, then the minimal value is attained at the boundary.
\begin{proof}
Suppose the minimizer is attained in $\beta$. We distinguish the two cases $\beta < \Phi^{-1}(\beta)$ and $\beta > \Phi^{-1}(\beta)$ and start with the first case. Suppose the smallest solution of $z = \Phi^{-1}(z)$ that is larger than $\beta$ is in $\gamma$. $\gamma$ exists because $1$ has that property. Then, because $\beta$ is a global minimizer, we have $h(\beta) \leq h(\gamma)$ and thus
 \begin{align*}
    0 \geq h(\beta) - h(\gamma) &=  -\int_{\beta}^{\gamma} \sign(z - \Phi^{-1}(z)) z~  dz  + \int_{\beta}^{\gamma} \sign(z - \Phi^{-1}(z)) (z- 1) dz \\
&+ |\beta - \Phi^{-1}(\beta)|  - |\gamma - \Phi^{-1}(\gamma)|\\
&=  \int_{\beta}^{\gamma}   z~  dz  -  \int_{\beta}^{\gamma} (z - 1)dz   + |\beta - \Phi^{-1}(\beta)| - |\gamma - \Phi^{-1}(\gamma)| \\
&= (\gamma - \beta) + |\beta - \Phi^{-1}(\beta)| - |\gamma - \Phi^{-1}(\gamma)|
 \end{align*}
This implies, recalling that $\gamma = \Phi^{-1}(\gamma)$,
$$ 0 = |\gamma - \Phi^{-1}(\gamma)|  \geq   |\beta - \Phi^{-1}(\beta)| + (\gamma - \beta)$$
which is a contradiction. Suppose now that the minimizer happens in $\beta$ and $\beta > \Phi^{-1}(\beta)$. Suppose the largest solution of $z = \Phi^{-1}(z)$ that is smaller than $\beta$ happens in $\alpha$. This $\alpha$ exists because $0$ has that property.  Then
 \begin{align*}
     0 \geq h(\beta) - h(\alpha) &=  \int_{\alpha}^{\beta} \sign(z - \Phi^{-1}(z)) z~  dz  - \int_{\alpha}^{\beta} \sign(z - \Phi^{-1}(z)) (z- 1) dz \\
&+ |\beta - \Phi^{-1}(\beta)| - |\alpha - \Phi^{-1}(\alpha)| \\
&= (\beta - \alpha) + |\beta - \Phi^{-1}(\beta)| - |\alpha - \Phi^{-1}(\alpha)|.
 \end{align*}
 Since $\alpha - \Phi^{-1}(\alpha) = 0$ and $\beta > \alpha$, this is a contradiction. 
\end{proof}
 
This simplifies the problem: the minimizer of the functional has to be a fixed point, a solution of $\alpha = \Phi(\alpha)$. We can thus drop the $|\alpha- \Phi(\alpha)|$ and consider the restricted optimization problem
$$ \min_{0 \leq \alpha \leq 1 \atop \alpha = \Phi(\alpha)} \int_0^{\alpha} \sign(z - \Phi^{-1}(z)) z~  dz  + \int_{\alpha}^1 \sign(z - \Phi^{-1}(z)) (z- 1) dz.$$

\subsection{Simplification II}
The result in the previous section shows that it is enough to consider fixed points $\Phi(\alpha) = \alpha$ and study the optimization problem
$$ \min_{0 \leq \alpha \leq 1 \atop \alpha = \Phi(\alpha)} \int_0^{\alpha} \sign(z - \Phi^{-1}(z)) z~  dz  + \int_{\alpha}^1 \sign(z - \Phi^{-1}(z)) (z- 1) dz.$$
This is a simpler functional but has a more complicated restriction, $\Phi(\alpha) = \alpha$. The purpose of this short section is to shows that this new functional has again a remarkable algebraic structure that ensures that any minimizer is automatically a fixed point. This allows us to go from restricted optimization of the simpler functional to unrestricted optimization of the simpler functional.

\begin{lemma}
If $\phi, \Phi$ are as above and if
\begin{align*}
    X &=  \min_{0 \leq \alpha \leq 1 \atop \alpha = \Phi(\alpha)} \int_0^{\alpha} \sign(z - \Phi^{-1}(z)) z~  dz  + \int_{\alpha}^1 \sign(z - \Phi^{-1}(z)) (z- 1) dz \\
    Y &=  \min_{0 \leq \alpha \leq 1} \int_0^{\alpha} \sign(z - \Phi^{-1}(z)) z~  dz  + \int_{\alpha}^1 \sign(z - \Phi^{-1}(z)) (z- 1) dz,
\end{align*}
  then we have $X=Y$. 
\end{lemma}
\begin{proof}
    If $Y$ attains its minimum in $\alpha = 0$ or $\alpha =1$, then there is nothing to prove. If the minimum is attained inside $(0,1)$ and not in fixed point $\Phi(\alpha) = \alpha$ then, because there are only finitely many fixed points, we are a positive distance from the nearest fixed point. Abbreviating
    $$ g(\alpha) =  \int_0^{\alpha} \sign(z - \Phi^{-1}(z)) z~  dz  + \int_{\alpha}^1 \sign(z - \Phi^{-1}(z)) (z- 1) dz$$
and differentiating leads to
\begin{align*}
    g'(\alpha) &= \sign(\alpha - \Phi^{-1}(\alpha)) \alpha -  \sign(\alpha - \Phi^{-1}(\alpha)) (\alpha - 1)  = \sign(\alpha - \Phi^{-1}(\alpha)).
\end{align*} 
We see that this is $\pm 1$ which ensures that we have not yet found the minimum.
\end{proof}

We see that this is $\pm 1$ unless we are in a fixed point. Therefore, the numerical value of the minimum is the same as the numerical value of the unrestricted problem
$$ \min_{0 \leq \alpha \leq 1} \int_0^{\alpha} \sign(z - \Phi^{-1}(z)) z~  dz  + \int_{\alpha}^1 \sign(z - \Phi^{-1}(z)) (z- 1) dz$$
and we may work with the unrestricted problem instead. Naturally, when actually trying to solve the problem, the restriction $\alpha = \Phi(\alpha)$ is very useful and should not be omitted. In contrast, when trying to prove that the minimum is quite negative somewhere, it is easier to work without the restriction.

\subsection{Simplification III}
We will now prove the main structural inequality. A variant of it first appeared in \cite{stein2} where it had the following form.

\begin{lemma}[\cite{stein2}] Let $g:[0,1] \rightarrow \mathbb{R}$ be continuous and $g \not\equiv 0$. Then 
$$\max_{0 \leq z \leq 1} ~  \int_0^z g(x) x ~dx +  \int_z^1 g(x) (x-1) ~dx \geq  \frac{1}{16} \frac{1}{ \| g\|_{L^2}^2} \left(\max_{0 \leq x \leq 1} \left| \int_0^x g(y) dy\right|\right)^3.$$
\end{lemma}
Taking $g$ to be the characteristic function on $[0, \varepsilon]$ shows that the inequality is sharp up to constants. The setting in \cite{stein2} required $g:[0,1] \rightarrow \mathbb{R}$ to be a function that could attain all possible real values and this allows for examples that are supported on very small intervals. Our setting is much nicer insofar as we can restrict ourselves to studying the same problem for functions $g:[0,1] \rightarrow \left\{-1,1\right\}$ which excludes these counterexamples and allows for results that have a more concrete interpretation in terms of the geometry of the cumulative distribution function.

\begin{lemma} Let $g:[0,1] \rightarrow \left\{-1,1\right\}$ with finitely many sign changes. Then
    $$ -X = \min_{0 \leq \alpha \leq 1} \int_0^{\alpha} g(z) z~  dz  + \int_{\alpha}^1 g(z) (z- 1) dz \leq 0.$$
    Using $\ell_1, \ell_2,  \dots$ to denote the length of the intervals on which $g$ is constant,  
    \begin{enumerate}
        \item the presence of a long interval ensures that $X$ is large
        $$ X \geq \max_{k \in \mathbb{N}} ~\frac{\ell_k^2}{4}.$$
        \item there cannot be too many intervals with lengths much larger than $X$
        $$ \sum_{\ell_k > X} (\ell_k - X)^2 \leq 2X$$
        \item if $g$ has $S$ sign changes, then
        $$  X \geq \frac{1}{4S}.$$
    \end{enumerate}
\end{lemma}
\begin{proof} We first prove (1), then give an independent argument for (2) which will imply (3). We abbreviate
$$ h(\alpha) =  \int_0^{\alpha} g(z) z~  dz  + \int_{\alpha}^1 g(z) (z- 1) dz.$$
The same computation as above shows that whenever $h'(\alpha)$ exists, it satisfies $h'(\alpha) \in \left\{-1, +1\right\}$. We first rewrite the function as
$$ h(\alpha) = \int_0^1 g(z) z~dz - \int_{\alpha}^1 g(z) dz.$$
Introducing the anti-derivative 
$$ G(z) = \int_0^z g(y) dy$$
and integrating by parts, we see that
\begin{align*}
     \int_0^1 g(z) z ~dz &= G(z) z \big|_{0}^1 - \int_0^1 G(z) dz = G(1) - \int_0^1 G(z) dz.
\end{align*}
Combining this with $\int_{\alpha}^1 g(z) dz = G(1) - G(\alpha)$, we arrive that
$$ h(\alpha) = G(\alpha) - \int_0^1 G(x) dx.$$
This representation shows that $h$ has mean value 0 since
$$ \int_0^1 h(\alpha) d\alpha = \int_0^1 G(\alpha) d\alpha -  \int_0^1 G(x) dx = 0.$$
A non-constant function assumes values both above and below the mean value and $h$ is not constant because $h'(\alpha) \in \left\{-1,1\right\}$ for all but finitely many values. Therefore
  $$ -X = \min_{0 \leq \alpha \leq 1} \int_0^{\alpha} g(z) z~  dz  + \int_{\alpha}^1 g(z) (z- 1) dz$$
satisfies $X >0$. It remains to make this step quantitative.\\

Suppose that $I_k \subset [0,1]$ is an interval on which $g$ is constant and suppose that interval has length $|I_k| = \ell_k$. Then 
$$ \max_{a \in I_k} h(a) - \min_{a \in I_k} h(a) = \ell_k.$$
Therefore
$$ \max_{a \in I_k} h(a) \geq \ell_k - X.$$
Provided $\ell_k \geq X$, we can deduce that, since $h$ is $1-$Lipschitz, one has
$$ \int_{0}^{1} \max\left\{ h(x), 0 \right\} dx \geq \frac{1}{2}(\ell_k - X)^2.$$
Simultaneously, we have the trivial estimate
$$ \int_{0}^{1} \min\left\{ h(x), 0 \right\} dx \geq -X$$
and the fact that $h$ has mean value 0 then implies that
$$ \frac{1}{2}(\ell_k -X)^2 - X \leq 0 \qquad \mbox{meaning that} \qquad (\ell_k -X)^2 \leq 2X.$$
If $X \geq 1/4$, then we are done since $\ell_k \leq 1$. It remains to deal with the case $X \geq 1/4$.
Assuming $\ell_k \geq X$, this implies, using $\sqrt{2X} + X \leq 2 \sqrt{X}$ for all $0 \leq X \leq 0.35$ that
$$ \ell_k \leq \sqrt{2X} + X \leq 2\sqrt{X} \qquad \mbox{and thus} \qquad X \geq \frac{\ell_k^2}{4}.$$
Therefore we either have $\ell_k \leq X$, in which case we certainly also have $\ell_k^2/4 \leq X$ because $\ell_k \leq 1$, or we have $\ell_k \geq X$ in which case we have $X \geq \ell_k^2/4$. This means we have $X \geq \ell_k^2/4$ in either case. This proves (1).\\

As for the second statement, we argue as above and note that if $I_k \subset [0,1]$ is an interval on which $g$ is constant and if it has length $|I_k| = \ell_k$, then  
$$ \max_{a \in I_k} h(a) \geq \ell_k - X.$$
If $\ell_k > X$, then 
$$ \int_{I_k} \max\left\{h(x), 0\right\} dx \geq \frac{(\ell_k - X)^2}{2}.$$
Now since,
\begin{align*}
    0 &= \int_0^1 h(x) dx = \int_0^1 \min\left\{h(x),0\right\}  dx + \int_0^1 \max\left\{h(x),0\right\}  dx \\
    &\geq -X + \int_0^1 \max\left\{h(x),0\right\}  dx,
\end{align*}
we may deduce that
$$ \sum_{\ell_k > X} \frac{(\ell_k - X)^2}{2} \leq \int_0^1 \max\left\{h(x),0\right\}  dx \leq X.$$
Suppose now there are only $S - 1$ sign changes. This leads to $S$ intervals. Maybe only very few, say $F \leq S$ of them, have length larger than $X$. Then the $S-F$ intervals have length $\leq X$ and we may deduce
$$ \sum_{\ell_k > X} \ell_k \geq 1 - (S-F)X.$$
Note that the $(S-F)X$ short interval cannot exceed length 1 and thus this lower bound is nonnegative: $1 - (S-F)X \geq 0$. The Cauchy-Schwarz inequality implies
\begin{align*}
     0\leq \left(\frac{1}{F} \sum_{\ell_k > X} \ell_k\right)  -  X &= \frac{1}{F} \sum_{\ell_k > X} (\ell_k - X) \\
    &\leq \frac{1}{F} \left(\sum_{\ell_k > X} (\ell_k - X)^2 \right)^{1/2} \sqrt{F}
\end{align*}
Therefore
$$ 2X \geq \sum_{\ell_k > X} (\ell_k - X)^2 \geq \frac{1}{F} \left[\sum_{\ell_k > X} (\ell_k  -  X) \right]^2.$$
We have
$$ \sum_{\ell_k > X} (\ell_k - X) \geq 1 - SX.$$
If $X \leq 1/(2S)$, then 
$$ 2X \geq \frac{1}{F} \left[\sum_{\ell_k > X} (\ell_k  -  X) \right]^2 \geq \frac{1}{4F} \geq \frac{1}{4S}.$$
This implies the desired result.
\end{proof}

\section{A related sequence}

\subsection{A periodic discrepancy energy and numerics}
We point out a related sequence that seems to be roughly equally good at carrying out the desired task of updating an exisiting set of points in $[0,1]$ to approximate the uniform distribution as quickly and uniformly as possible. It is remarkable that this sequence seems to arise independently in several different papers in the literature.
The method is defined as follows: introduce the $1-$periodic function $p:\mathbb{R} \rightarrow \mathbb{R}$ which satisfies
$$ p(x) =  x^2 - x + \frac{1}{6} \qquad \mbox{for}~0 \leq x \leq 1$$
and is then extended periodically. This is the second Bernoulli polynomial. Given $0 \leq x_1, \dots, x_n \leq 1$, the new point is defined via
$$ x_{n+1} =  \arg\min_{0 \leq x \leq 1} \sum_{k=1}^{n} p(x - x_k).$$
This type of construction arises in at least four different ways (which, ultimately, are all equivalent in one form or another)
\begin{itemize}
    \item as a way to greedily minimize Zinterhof's diaphony \cite{zinterhof} which is originally defined as a Fourier multiplier (see also Lev \cite{lev})
    \item in a 2014 paper of Hinrichs and Oettershagen \cite{hinrichs} where it arises as the reproducing kernel for functions in the Sobolev space $H^1(\mathbb{T}^2)$. They show that if $\left\{x_1, \dots, x_n\right\} \subset [0,1]$ is given, then the periodic $L^2-$discrepancy can be written as 
    $$ L_{2}^{\mbox{\tiny per}} = - \frac{1}{3} + \frac{1}{n^2} \sum_{i,j=1}^{n} \left(\frac{1}{3} + p(x_i - x_j)\right).$$
    \item in a 2020 paper of Brown and Steinerberger \cite{brown2} where it arises as a linearization of the Wasserstein distance $W_2$ or, alternatively, from the identity
    $$ 2 \pi^2 \sum_{0 \neq h \in \mathbb{Z}} \frac{e^{2\pi i h x}}{h^2} = \left\{x\right\}^2 - \left\{x\right\} + \frac{1}{6}$$
    \item and in a completely different way, from a regularity measure proposed by Wagner \cite{wagner}, see below for a discussion. 
\end{itemize}

We were motivated by a Lemma of Wagner \cite{wagner}: consider the $1-$periodic function
$ h(x) = 1/2 - \left\{ x \right\}$
and introduce the corresponding function
$$ \Delta(x) = \sum_{i=1}^{n} h(x - x_i).$$

\begin{lemma}[Wagner \cite{wagner}] The function $\Delta$ is piecewise linear with slope $-N$ and jumps of height 1 at each point. It has mean value 0
$$ \int_0^1 \Delta(x) = 0.$$
    It is connected to the (extreme) discrepancy $D_N$ via
    $$ \frac{1}{2} D_N \leq \max_{0 \leq x \leq 1} |\Delta(x)| \leq D_N.$$
\end{lemma}
$\Delta(x)$ is also shown to have a number of other structural properties (that will not be of importance here, we refer to \cite{wagner}). Motivated by this nice equivalence, one may be tempted to add the next point $x_{n+1}$ so
as to minimize the expression
$$ \int_0^1 \Delta_{n+1}(x)^2~ dx \rightarrow \min.$$
An explicit computation shows that 
\begin{align*}
    \int_0^1 \left( \sum_{i=1}^{n+1} h(x - x_i) \right)^2 = \int_0^1 \Delta_n(x)^2 + \int_0^1 h(x-x_{n+1})^2 dx + 2\int_0^1 \Delta_n(x) h(x-x_{n+1})
\end{align*}
implying that one should set the next point $x_{n+1}$ so that
 $$ \int_0^1 \Delta_n(x) h(x-x_{n+1}) dx \rightarrow \min.$$
This can be further simplified. One has
\begin{align*}
     \int_0^1 \Delta_n(x) h(x-x_{n+1}) dx &=  \int_0^1 \sum_{k=1}^{n} \left(\frac12 - \left\{x - x_k \right\} \right) \left(\frac12 - \left\{x - x_{n+1} \right\} \right) dx \\
     &=  \sum_{k=1}^{n} \int_0^1  \left(\frac12 - \left\{x - x_k \right\} \right) \left(\frac12 - \left\{x - x_{n+1} \right\} \right) dx.
\end{align*}
  A short computation shows
$$ \int_0^1  \left(\frac12 - \left\{x - x_k \right\} \right) \left(\frac12 - \left\{x - x_{n+1} \right\} \right) dx = p(x_k - x_{n+1})$$
where $p(x)$ is the second Bernoulli polynomial. This is exactly the sequence above.

\subsection{Numerical results}
We return to the special case of trying to find a sequence of points in $[0,1]$ such that
$\left\{x_1, \dots, x_N\right\}$ is as uniformly distributed as possible \textit{uniformly} in $N$. This is a very difficult problem in its own right. The two metrics most commonly used in the area are the star-discrepancy $D_N^*$
$$ D_N^* = \max_{0 \leq x \leq 1} \left| \frac{\# \left\{1 \leq i \leq N: x_i \leq x \right\}}{N} - x \right|$$
and the periodic discrepancy (also known as extreme discrepancy in one dimension)
$$ D_N = \max_{J \subset [0,1] \atop J~{\tiny \mbox{interval}}} \left| \frac{\# \left\{1 \leq i \leq N: x_i \in J \right\}}{N} - |J| \right|.$$
These two quantities are related and satisfy $D_N^* \leq D_N \leq 2D_N^*$. 
It is known that the best uniform estimate one can hope for asymptotically is \cite{schm}
$$ D_N^* \leq c \frac{\log{N}}{N}$$
where the constant $c$ can be chosen as small as $\sim 0.22$ \cite{faure, ost} 
and cannot be smaller than $\sim 0.065$ \cite{larcher, puch}. As is customary in such situations, the upper bound is probably closer to the truth. This scaling suggests that the appropriate scale-invariant quantities are $N D_N^*/\log{N}$ and $N D_N/\log{N}$.
Classic examples that are known to attain a very good rate are the van der Corput sequence~\cite{vdc} and the Kronecker sequence with golden ratio. We refer to the Kronecker sequence as the Fibonacci sequence, it is given by $\left(\{\varphi n\}\right)_{n \in \mathbb{N}}$, where $\{\varphi n\}$ is the fractional part of golden ratio times $n$. A greedy contender is the much more recent Kritzinger sequence~\cite{kritzinger} for which one of the authors \cite{clement} has shown numerically that it performs exceedingly well. We start by comparing with the `periodic energy' construction introduced in \S 5.1 with initial values $\left\{1/3, 1/2\right\}$. The results are shown in Figure~\ref{fig:perio}. We see that the sequence from $\S 5.1$ performs well and is, within a small constant, comparable to the leading constructions. 
\vspace{-20pt}
\begin{figure}[h]
    \centering
    \includegraphics[width=0.49\linewidth]{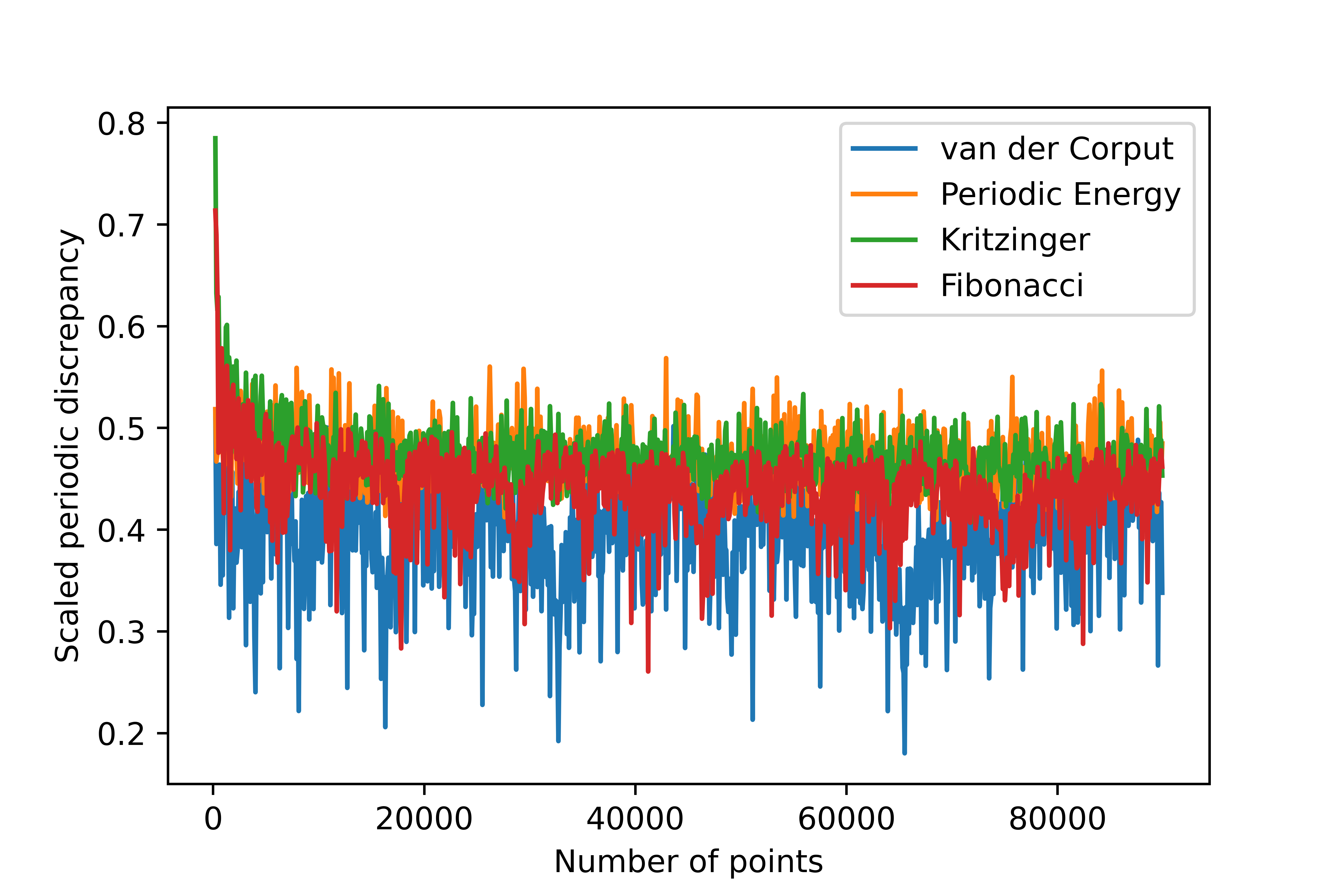}
    \hfill
    \includegraphics[width=0.49\linewidth]{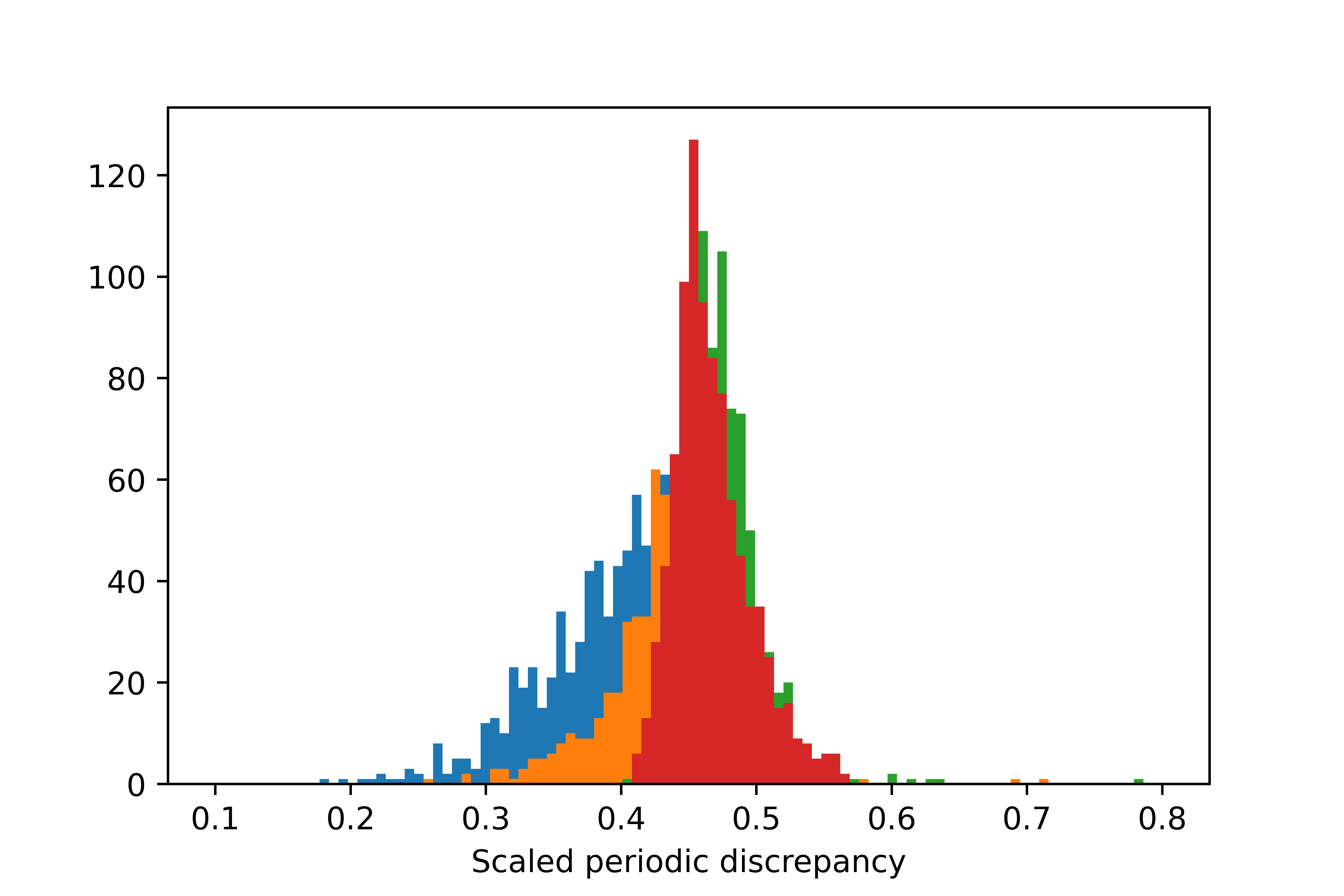}
    \caption{Left: $N D_N/\log{N}$ for three excellent sequences and the new sequences from \S 5.1. Right: a histogram of the values.}
    \label{fig:perio}
\end{figure}

As a second experiment, we consider the main object in this paper. For the purpose of excellent numerics, we replace the energy
$$\sum_{i=1}^n\left|x_i-\frac{i}{n}\right| \qquad \mbox{by the closesly related energy} \qquad \sum_{i=1}^n\left|x_i-\frac{2i-1}{2n}\right|.$$
From a mathematical perspective, these two objects are very nearly indistinguishable and the first one is a bit more convenient to work with. 
\vspace{-10pt}
\begin{figure}[h!]
    \centering
    \includegraphics[width=0.9\linewidth]{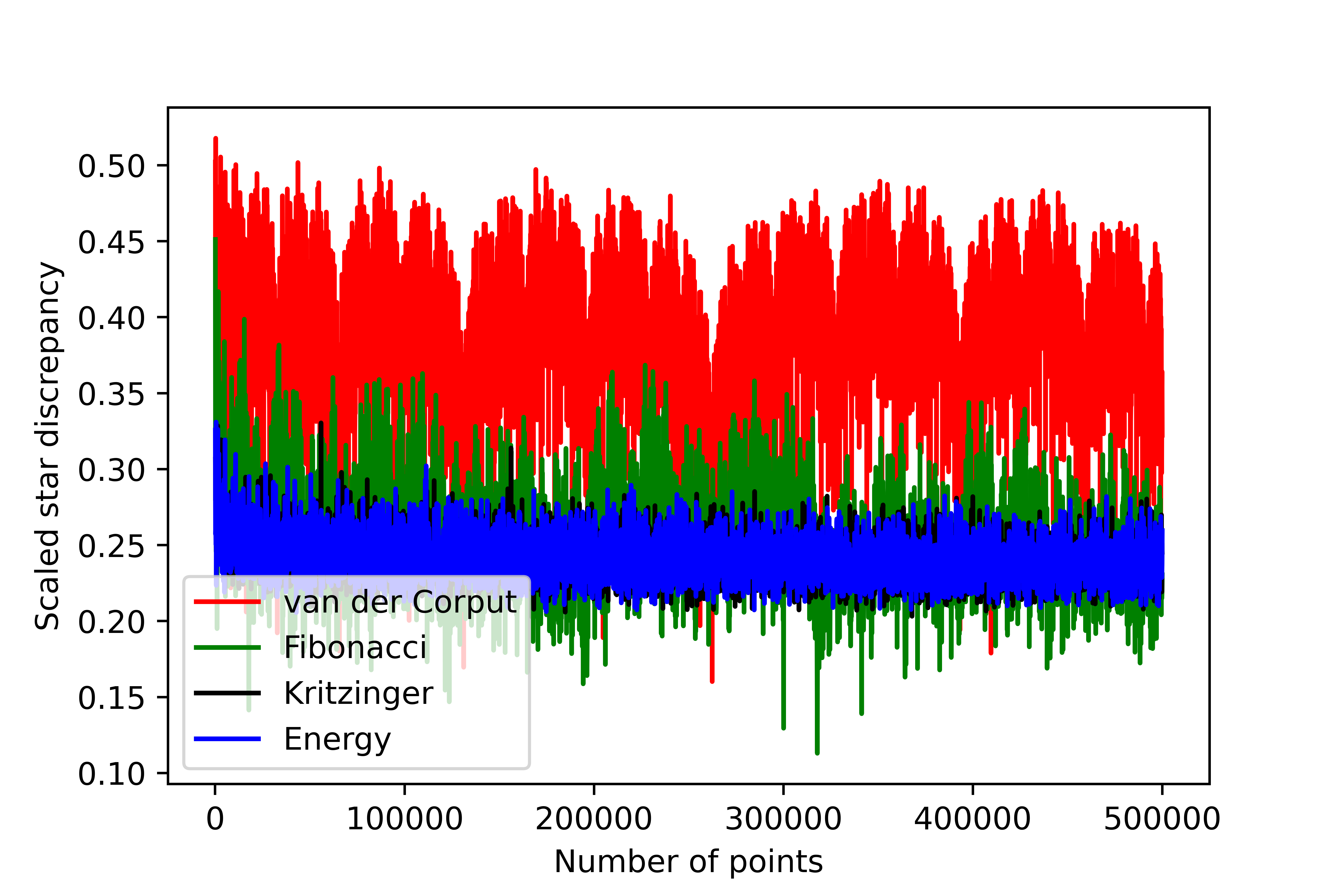}
    \caption{The values of $N D_N^*/\log{N}$ Fibonacci, van der Corput, Kritzinger and Energy sequences. Energy and Kritzinger are extremely similar and, on average, the best performers.}
    \label{fig:star}
\end{figure}
\vspace{-10pt}
This is motivated by the fact that the set $\left\{(2i-1)/(2n): i\in \{1,\ldots,n\}\right\}$ is known to be optimal for the star discrepancy in one dimension~\cite{kuipers}. We use a very close variant of the algorithm described in \S 2.2 to obtain the sequence. 
Figure~\ref{fig:star} compares the star discrepancy of this sequence with state-of-the-art constructions. It performs as well as the Kritzinger sequence, which was shown to empirically outperform the Fibonacci sequence on average~\cite{clement}. We believe these results to be fascinating for a number of different reasons, including

\begin{enumerate}
    \item the empirical performance of greedy energy minimization is as good (if not better) than the known classical constructions
    \item our construction and the Kritzinger construction have comparable behavior and appear to still be (albeit very slowly decaying); this suggests that there is a common underlying theory or, alternatively,
    \item that possibly both of these sequences are either asymptotically optimal or close to optimal in some suitable sense.
\end{enumerate}

All of this suggests we are possibly reaching the limit of what can be obtained via greedy constructions, and possibly for one-dimensional sequences in general. Needless to say, a theoretical explanation at that level of regularity is still missing.

\end{document}